\documentclass[12pt]{article}
\usepackage{amsmath,amsfonts,amsthm,amssymb, graphicx,bbm,ulem}
\usepackage[usenames,dvipsnames]{xcolor}
\usepackage[left=1in,top=1in,right=1in]{geometry}

\newcommand{\prob}{\mathbb{P}}
\newcommand{\pe}{\mathbb{E}}

\newcommand{\edge}{\mathcal{E}}

\newcommand{\geo}{\overline{\textsc{Geo}}}

\newtheorem{thm}{Theorem}
\newtheorem{lem}[thm]{Lemma}
\newtheorem{prop}[thm]{Proposition}
\newtheorem{clam}[thm]{Claim}
\newtheorem{defin}{Definition}

\newtheorem{ass}{Assumption}
\newtheorem{rem}{Remark}

\numberwithin{equation}{subsection}

\title{Lower bounds for fluctuations in first-passage percolation for general distributions}

\author{Michael Damron \\ Georgia Tech \and Jack Hanson \\ City College, CUNY \and Christian Houdr\'e \\ Georgia Tech \and Chen Xu \\ Uber Technologies}

\begin{document}

\maketitle

\abstract{In first-passage percolation (FPP), one assigns i.i.d.~weights to the edges of the cubic lattice $\mathbb{Z}^d$ and analyzes the induced weighted graph metric. If $T(x,y)$ is the distance between vertices $x$ and $y$, then a primary question in the model is: what is the order of the fluctuations of $T(0,x)$?
It is expected that the variance of $T(0,x)$ grows like the norm of $x$ to a power strictly less than 1, but the best lower bounds
available are (only in two dimensions) of order $\log \|x\|$. This result was found in the '90s and there has not been any improvement since. In this paper, we address the problem of getting stronger fluctuation
bounds: to show that $T(0,x)$ is with high probability not contained in an
interval of size $o(\log \|x\|)^{1/2}$, and similar statements for FPP in
thin cylinders. Such statements have been proved for special edge-weight
distributions, and here we obtain such bounds
for general edge-weight distributions. The methods involve inducing a fluctuation in the number of edges in a box whose weights are of ``hi-mode'' (large).}

\section{Introduction}

\subsection{Background}

We will consider first-passage percolation (FPP) on $\mathbb{Z}^2$ (or more generally $\mathbb{Z}^d$ with $d \geq 2$), with set $\mathcal{E}^2$ of nearest-neighbor edges. This means that we take a collection $(t_e)_{e \in \mathcal{E}^2}$ of i.i.d.~nonnegative random variables (passage times) and, for any lattice path $\Gamma$ (alternating sequence $x_0, e_0, x_1, e_1, \dots, x_{n-1},e_{n-1},x_n$ of vertices and edges such that $e_i$ has endpoints $x_i,x_{i+1}$), we assign the passage time $T(\Gamma) = \sum_{i=0}^{n-1} t_{e_i}$. Last, for vertices $x,y \in \mathbb{Z}^2$, we put
\[
T(x,y) = \inf_{\Gamma : x \to y} T(\Gamma),
\]
where the infimum is over all paths from $x$ to $y$.

$T$ as defined is a pseudometric, and FPP is the study of the random metric space $(\mathbb{Z}^2, T)$. A primary question involves the order of the variable $T(0,x)$, when $\|x\|$ is large. Under mild conditions on the distribution of $(t_e)$, one can show that $T(0,x)$ grows linearly: there is a norm $g$ on $\mathbb{R}^2$ such that 
\[
\lim_{\|x\| \to \infty} \frac{|T(0,x) - g(x)|}{\|x\|} = 0 \text{ a.s.},
\]
where $\|x\|$ is the Euclidean norm of $x$. (See Section~2.3 of \cite{ADH_book}.) This can be interpreted as
\[
T(0,x) = g(x) + o(\|x\|) \text{ as } \|x\| \to \infty.
\]
This leads to another important question: what is the order of the error term $o(\|x\|)$?

The error term splits into a ``random fluctuation'' term $T(0,x) - \mathbb{E}T(0,x)$ and a ``nonrandom fluctuation'' term $\mathbb{E}T(0,x) - g(x)$. Bounds on the latter (see \cite{Alexander97, ADH_gamma, Tessera}) typically use bounds on the former, so we will focus on the random fluctuation term. A typical way to measure its size is to estimate the variance of $T(0,x)$. Upper bounds on the variance were given by Kesten (order $\|x\|$ in \cite{Kesten_speed}) and later improved to $\|x\|/\log \|x\|$ in a series of works by Benjamini-Kalai-Schramm \cite{BKS}, Bena\"im-Rossignol \cite{BR}, and Damron-Hanson-Sosoe \cite{DHS15}. All of these works were for general dimensions $d \geq 2$. These are far from the predicted value of $\mathrm{Var}~T(0,x) \sim \|x\|^{2/3}$ in two dimensions, which has been verified in a related exactly-solvable model \cite{J}.

Lower bounds on the variance are less developed. In the '90s, Pemantle-Peres \cite{PP} (for exponential distribution) and Newman-Piza \cite{NP} (for general distributions, extended by Auffinger-Damron in \cite{AD11}) showed inequalities of the form
\[
\mathrm{Var}~T(0,x) \geq c \log \|x\|, \text{ for } x \neq 0,
\]
in two-dimensional FPP.  (The best lower bound for general dimensions remains of constant order \cite{Kesten_speed}.) Although the work \cite{NP} uses a martingale method which only gives a variance bound, the other \cite{PP} shows a stronger property of the distribution of $T(0,x)$. Specifically, they show that it cannot be supported on an interval of size $o(\sqrt{\log \|x\|})$: for any intervals $[a_n,b_n]$ with $b_n-a_n = o(\sqrt{\log n})$, one has $\mathbb{P}(T(0,nv) \in [a_n,b_n]) \to 0$ for any fixed unit vector $v$. Their proof uses the memoryless property of the exponential distribution to exhibit the growth of the ball $B(t) = \{x \in \mathbb{Z}^2 : T(0,x) \leq t\}$ as a Markov process.

A general method was given by Chatterjee in \cite{C17} to prove ``fluctuation lower bounds'' for a range of statistical physics models, so long as the underlying randomness lies in a certain class. For FPP, he has shown that under a strong regularity assumption on the distribution of $(t_e)$ (requiring $t_e$ to be a continuous random variable with a smooth density and rapidly decaying tails), the following holds. If $(y_n)$ is a sequence of points such that $\|y_n\|$ grows like a constant multiple of $n$, then for $d=2$, there exist $c_1,c_2>0$ such that for all large $n$, and all $-\infty < a \leq b < \infty$ with $b-a \leq c_1 \sqrt{\log n}$, one has $\mathbb{P}(a \leq T(0,y_n) \leq b) \leq 1-c_2$. Note that results of this type do not in general follow from variance lower bounds.

In this paper, we aim to improve the results of Chatterjee and Pemantle-Peres to general distributions. Our main results below apply to the largest class of distributions for which the FPP model does not exhibit degenerate behavior. In contrast to \cite{PP} and \cite{C17}, our methods do not rely on specific properties of the underlying distribution, and explore a more combinatorial avenue. The main idea is inspired by the study of the longest common subsequence problem (see \cite{gong2018lower, HoudreLCSVARLB2012, lember2009standard}), and involves introducing a fluctuation in the number of hi-mode weights (weights in the top part of the distribution of $t_e$). The notion of hi-mode weights was introduced and used in the Ph.D. thesis of Xu \cite{chen_thesis}, where lower bounds of order $n^{\frac{r(1-\alpha)}{2}}$ are obtained for the $r$-th central moments $(r \geq 1)$ in a related last-passage percolation model over an $n \times \lfloor n^{\alpha} \rfloor$ grid. (See \cite{CX18}.)

\subsection{Main results}
For the statement of our results, we need some conditions on the common distribution function $F$ of $(t_e)$:
\begin{equation}\label{eq: perc_assumption}
F(0)<p_c \text{ and } F(I) < \vec{p_c} \text{ if }I>0,
\end{equation}
where $p_c = 1/2$ is the critical value for two-dimensional bond percolation, $\vec{p_c}\sim 0.644$ is the critical value for two-dimensional oriented bond percolation (see \cite{durrett_oriented}), and $I$ is the infimum of the support of $F$. Note that \eqref{eq: perc_assumption} implies that $F$ is non-degenerate.
\begin{thm}\label{thm: main_thm}
Let $F$ be a distribution satisfying \eqref{eq: perc_assumption}. There exist families of reals $(A_x)_{x \in \mathbb{Z}^2}$ and $(B_x)_{x \in \mathbb{Z}^2}$ such that
\[
\liminf_{\|x\| \to \infty} \mathbb{P}(T(0,x) \leq A_x) > 0,
\]
\[
\liminf_{\|x\| \to \infty} \mathbb{P}(T(0,x) \geq B_x) > 0,
\]
and
\[
\liminf_{\|x\| \to \infty}\frac{B_x-A_x}{\sqrt{\log \|x\|}} > 0.
\]
\end{thm}

\begin{rem}
Here we discuss optimality of the condition \eqref{eq: perc_assumption}. If $I=0$ and $F(0)>p_c$, then there is an infinite component of zero-weight edges, and this makes $T(0,x)$ stochastically bounded in $x$, so Theorem~\ref{thm: main_thm} cannot hold. If $I=0$ and $F(0)=p_c$, then it was shown in \cite{DLW} that $\mathrm{Var}~T(0,x) \asymp \sum_{k=1}^{\lfloor \log \|x\|\rfloor} (F^{-1}(p_c + 1/2^k))^2$. Similar arguments will show that under no moment condition, Theorem~\ref{thm: main_thm} holds with $\sqrt{\log \|x\|}$ replaced by $\sqrt{\sum_{k=1}^{\lfloor \log \|x\|\rfloor} (F^{-1}(p_c+1/2^k)})^2$. If, instead, one has $I>0$ and $F(I) \geq \vec{p}_c$, then the limit shape for the model has flat segments (see \cite[Sec.~2.5]{ADH_book}), and for directions corresponding to these segments, $T(0,x)$ has fluctuations of order constant. For directions outside these segments, the arguments of this paper can be adapted (the proof of Lemma~\ref{lem: good_lemma_2} must be replaced with an analysis similar to that of \cite[Sec.~4]{AD11}) to show that under a moment condition on $t_e$, $T(0,x)$ still fluctuates at least with order $\sqrt{\log \|x\|}$.
\end{rem}

\begin{rem}
Our strategy for the proof of Theorem~\ref{thm: main_thm}, when applied in dimensions $d>2$, gives only a constant lower bound. See the discussion at the beginning of Section~\ref{sec: cylinders}.
\end{rem}

The next result concerns FPP in thin cylinders of $\mathbb{Z}^2$. For $\alpha > 0$, we define the restricted cylinder passage time between vertices $0$ and $x$ as
\[
T(0,x;\alpha) = \inf_{\stackrel{\Gamma: 0 \to x}{\Gamma \subset C(x;\alpha)}} T(\Gamma),
\]
where $C(x;\alpha)$ is the thin cylinder
\[
C(x;\alpha) = \{z \in \mathbb{R}^2 : dist(z,L_x) \leq \|x\|^\alpha\},
\]
$dist$ is the Euclidean distance, and $L_x$ is the line through $0$ and $x$.
\begin{thm}\label{thm: cylinders}
Let $F$ be a distribution satisfying \eqref{eq: perc_assumption}. There exist families of reals $(A_x)_{x \in \mathbb{Z}^2}$ and $(B_x)_{x \in \mathbb{Z}^2}$ such that
\[
\liminf_{\|x\| \to \infty} \mathbb{P}(T(0,x;\alpha) \leq A_x) > 0,
\]
\[
\liminf_{\|x\| \to \infty} \mathbb{P}(T(0,x;\alpha) \geq B_x) > 0,
\]
and
\[
\liminf_{\|x\| \to \infty} \frac{B_x-A_x}{\|x\|^{\frac{1-\alpha}{2}}} > 0.
\]
\end{thm}
Although the above result is stated for two dimensions, it extends to $\mathbb{Z}^d$ with $d \geq 3$ under the assumption $\mathbb{E}t_e<\infty$. (This is needed to apply an analogue of \cite[Theorem~1.5]{Marchand} in the proof of Lemma~\ref{lem: good_lemma_2}.) The corresponding exponent is $\frac{1-\alpha(d-1)}{2}$.

\begin{rem}
Theorem~\ref{thm: cylinders} has implications for the original (unconstrained) FPP model. Under the (unproven) assumption that the asymptotic shape satisfies the ``positive curvature inequality'' (see the lower bound of Equation~(2.28) in \cite{ADH_book}) in direction $u$, one can show (with an additional moment condition on $(t_e)$) that any geodesic between $0$ and $nu$ will, with high probability, be confined to $C(nu;\alpha)$ for any given $\alpha > 3/4$. In this setting, the restricted cylinder passage time $T(0,nu;\alpha)$ is equal to $T(0,nu)$, and we obtain fluctuations for $T(0,nu)$ of order at least $n^{\beta}$ for any $\beta < 1/8$. (This result is analogous to variance lower bounds provided by Newman-Piza \cite{NP} and is one manifestation of the relationship between the variance exponent and the transversal (geodesic) wandering exponent --- see \cite{Chatterjee,LNP}.) See also \cite{xxx} for a corresponding result in the aforementioned related last-passage percolation model.
\end{rem}

The proofs of both theorems are similar, and the second is even somewhat easier. For this reason, we will give the full proof of Theorem~\ref{thm: main_thm} in Section~\ref{sec: main_proof}, and sketch the proof of Theorem~\ref{thm: cylinders}, indicating the necessary adjustments to the first proof, in Section~\ref{sec: cylinders}.

\begin{rem}
While we were finalizing this paper, Bates and Chatterjee \cite{BC18} posted a paper to the arXiv containing results similar to those listed above, with substantially different proof methods. While our moment condition is weaker, their work also includes a study of fluctuations in polymer models (in addition to percolation models). We have not yet tried to extend our methods to other models.
\end{rem}

\section{Proof of Theorem~\ref{thm: main_thm}}\label{sec: main_proof}
Let us assume to aim for a contradiction:
\begin{ass}\label{ass: main_assumption}
We assume that there exist
\begin{enumerate}
\item a sequence $(w_n)$ of reals decreasing to 0,
\item a sequence $(a_n)$ of reals, and
\item a sequence of nonzero points $(x_n)$ in $\mathbb{Z}^2$ with $\|x_n\| \to \infty$ 
\end{enumerate}
such that if $J_n$ is defined as 
\[
J_n = \left[ a_n, a_n + w_n \sqrt{\log \|x_n\|}\right],
\]
then
\[
\mathbb{P}\left(T(0,x_n) \in J_n \right) \to 1 \text{ as } n \to \infty.
\]
\end{ass}


We will represent the passage time $T(0,x)$ as a function of three quantities. First, due to \eqref{eq: perc_assumption}, we may find $d_0 >I$ such that 
\begin{equation}\label{eq: d_0_def}
F(d_0) \in (0,1).
\end{equation}
Note that this implies that
\begin{equation}\label{eq: implication_d_0}
\mathbb{P}(t_e < d_0) > 0.
\end{equation}
\begin{defin}
Any weight $t_e$ with $t_e \leq d_0$ is called lo-mode, and all other weights are called hi-mode. 
\end{defin}
For $K>0$ large and to be determined later, we define the box $\mathcal{B}(j)$ for $j \geq 1$ to be the set of edges with both endpoints in $[-K^j,K^j]^2$, and the annuli
\begin{equation}\label{eq: annulus_def}
A(1) = \mathcal{B}(1) \text{ and } A(j) = \mathcal{B}(j) \setminus \mathcal{B}(j-1) \text{ for } j \geq 2.
\end{equation}
Let $\vec{N} = (N_1, N_2, \dots)$ be a vector with independent entries, such that $N_j$ is binomial with parameters $\#A(j)$ and $1-F(d_0)$. $N_j$ will represent the number of edges with hi-mode weights in $A(j)$. Let $\Pi = (\pi_1, \pi_2, \dots)$ be a vector of independent uniform permutations, where $\pi_j$ is an ordering of the elements of $A(j)$. In other words, $\pi_j$ is a (uniformly chosen) bijection from $A(j)$ to $\{1, 2,\, \ldots,\, \#A(j)\}$; we write the image of an element $e \in A(j)$ as $\pi_j(e)$. Last, let $P = (t_e^{(L)},t_e^{(H)})_{e \in \mathcal{E}^2}$ be a collection of i.i.d.~pairs of weights assigned to each edge, with $t_e^{(L)}$ and $t_e^{(H)}$ independent and distributed as $t_e$ conditional on $\{t_e \leq d_0\}$ and $\{t_e > d_0\}$ respectively. Given these variables (which we assume all to be mutually independent), we define an edge-weight configuration $(t_e)$ by setting $t_e = t_e^{(H)}$ if $e \in A(j)$ and $\pi_j(e) \leq N_j$, and $t_e = t_e^{(L)}$ otherwise. Note that $(t_e)$ is then distributed as our original edge-weights were, and so from now on we will use only these $t_e$'s. We can then think of the passage time $T(0,x)$ as a function of the triple $(\vec{N},\Pi, P)$.

\subsection{Outline of the proof}

The main idea of the proof is to examine the effect of changing the vector $\vec{N}$, which records the number of hi-mode weights in the system. In each annulus $A(k)$, the variable $N_k$ is likely to fluctuate on the order of the square root of the volume of $A(k)$, and this fluctuation changes the passage time $T(0,x_n)$ by at least a constant. Because weights of edges in distinct annuli are independent, the total change in $T(0,x_n)$, summed over all $A(k)$ from $k=1$ to $C\log\|x_n\|$ should be of order $\sqrt{\log \|x_n\|}$.

To make the above idea rigorous, we attempt to show a ``small ball'' probability bound for $T(0,x_n)$ of the form
\[
\sup_r \mathbb{P}(T(0,x_n) \in [r,r+\epsilon] ) \leq \frac{C}{\sqrt{\log \|x_n\|}}.
\]
If we were able to do this, then we could immediately contradict Assumption~\ref{ass: main_assumption}. Unfortunately, we can only prove such a small ball bound for the conditional expectation $\mathbb{E}[T(0,x_n) \mid \vec{N}]$, and this bound does not directly contradict the assumption. To obtain a contradiction from it, we need to know that $T(0,x_n)$ has sufficiently quickly decaying tails. For this reason, we start by making a truncation $T_n$ of $T(0,x_n)$ with the following properties:
\begin{enumerate}
\item $T_n \in [A_n, A_n + \sqrt{\log \|x_n\|}]$ a.s.~for some $A_n \in \mathbb{R}$, and
\item the midpoint $M_n$ of the above interval is a median of $\mathbb{E}[T_n \mid \vec{N}]$.
\end{enumerate}
This truncation is done in step 1 (Section~\ref{sec: step_1}).

In step 2 (Section~\ref{sec: step_2}), we prove the small ball probability bound for the conditional expectation of the truncated passage time, $\mathbb{E}[T_n \mid \vec{N}]$ (see Proposition~\ref{prop: small_ball}). The proof involves encoding the $N_j$'s using a sequence of i.i.d.~Bernoulli random variables, and analyzing them combinatorially. The event that $\mathbb{E}[T_n \mid \vec{N}]$ lies in an interval $[r,r+\epsilon]$ is shown (in Lemma~\ref{lem:showanti}) to be a particular type of subset of the hypercube known as an antichain. We can then apply Sperner's theorem on the size of antichains to obtain the small ball result. The antichain lemma, Lemma~\ref{lem:showanti}, only applies to a certain set of ``good'' values of $\vec{N}$ which is defined in step 1 (see Definition~\ref{def: good}). A large part of step 1 is focused on showing that this set has high probability (see Proposition~\ref{prop: good}), and involves new geodesic estimates under no moment condition (Lemma~\ref{lem: good_lemma_2}).

The last step, step 3 (Section~\ref{sec: step_3}), uses the small ball result of step 2 to show that one has
\[
\mathbb{P}(T(0,x_n) \leq M_n - c_1\sqrt{\log \|x_n\|}) > c_1,
\]
and
\[
\mathbb{P}(T(0,x_n) \geq M_n + c_1\sqrt{\log \|x_n\|}) > c_1,
\]
for some $c_1>0$. This is done in Proposition~\ref{prop: contradiction}, and contradicts Assumption~\ref{ass: main_assumption}.

\subsection{Step 1. Truncation and definition of good $\vec{N}$'s.} \label{sec: step_1}

First we make a particular truncation of $T(0,x_n)$. This is done so that estimates on $\mathbb{E}[T_n \mid \vec{N}]$ can be brought back to $T(0,x_n)$ in step 3.

\begin{lem}\label{lem: truncation}
For each $n$, there exists a real $A_n$ such that if $B_n := A_n + \sqrt{\log\|x_n\|}$ and
\[
T_n := \begin{cases}
A_n & \quad \text{if } T(0,x_n) \leq A_n \\
B_n & \quad \text{if } T(0,x_n) \geq B_n \\
T(0,x_n) & \quad \text{otherwise},
\end{cases}
\]
then some median of $\mathbb{E}[T_n \mid \vec{N}]$ is equal to $(A_n+B_n)/2$.
\end{lem}
\begin{proof}
If $\|x_n\|=1$, then the claim holds for any value of $A_n$, so we assume $n$ is large enough so that $\|x_n\|>1$. Because $n$ will not vary in the proof and $T_n$ is a function of $A_n$, we write $A_n = A$ and
\[
X_A = \frac{\mathbb{E}[T_n \mid \vec{N}]-A}{\sqrt{\log \|x_n\|}} = \mathbb{E}\left[ \frac{T_n - A}{\sqrt{\log \|x_n\|}} ~\bigg|~ \vec{N} \right].
\]
Then the family $(X_A)_{A \in \mathbb{R}}$ has the following properties:
\begin{enumerate}
\item $X_A$ takes values in $[0,1]$.
\item $X_A$ is non-increasing; that is, for $A \leq A'$, one has $X_A \geq X_{A'}$ a.s.
\begin{proof}
As a function of $A$, $T_n= T_n(A)$ is linear (with slope 1) for $A \leq T(0,x_n) - \sqrt{\log \|x_n\|}$, constant (slope 0) for $A \in [ T(0,x_n)  - \sqrt{\log \|x_n\|}, T(0,x_n) ]$, and linear (slope 1) for $A \geq T(0,x_n)$. Therefore $T_n - A$ is non-increasing and so is $X_A$.
\end{proof}
\item $\lim_{A\to-\infty} X_A = 1$ and $\lim_{A\to\infty} X_A = 0$ in probability.
\begin{proof}
Because $(T_n-A)/\sqrt{\log \|x_n\|}$ is equal to $1$ for $A$ negative enough and $0$ for $A$ positive enough, the bounded convergence theorem implies the statement.
\end{proof}
\item If $A_k \to A$ as $k \to \infty$, then $X_{A_k} \to X_A$ a.s.
\begin{proof}
Because $T_n=T_n(A)$ is continuous in $A$, so is $(T_n-A)/\sqrt{\log \|x_n\|}$. If $A_k \to A$, then $X_{A_k} \to X_A$ a.s.~by the bounded convergence theorem.
\end{proof}
\end{enumerate}

We now must show that for some $A$, $X_A$ has a median equal to $1/2$. Let $\mathcal{M}_A$ be the set of medians of $X_A$ and let $\mathcal{M} = \cup_A \mathcal{M}_A$. We first show that 
\begin{equation}\label{eq: step_1_median_a}
\mathcal{M}\cap [0,1/2) \neq \emptyset
\end{equation}
and
\begin{equation}\label{eq: step_1_median_b}
\mathcal{M} \cap (1/2,1] \neq \emptyset.
\end{equation}
The proofs of these are identical up to symmetry, so we only prove the first. By property 3, $\lim_{A \to \infty} X_A = 0$ in probability, so for large $A$, we have $\mathbb{P}(X_A \leq 1/4) > 1/2$. For such $A$, all medians of $X_A$ must lie in $[0,1/4]$, and so $\mathcal{M}_A \cap [0,1/2) \neq \emptyset$.

Next we show that if $(A_k)$ is a sequence of reals converging to a finite $A$ and $m_k$ is any median of $X_{A_k}$ such that $m_k \to m$ for some $m$, then
\begin{equation}\label{eq: median_converge}
m \text{ is a median of } X_A.
\end{equation}
To see why this is true, assume that $m$ is not a median of $X_A$. Then either $\mathbb{P}(X_A \leq m) < 1/2$ or $\mathbb{P}(X_A \geq m) < 1/2$. In the first case, pick $\delta>0$ such that $\mathbb{P}(X_A \leq m+\delta) < 1/2$ and $m+\delta$ is not an atom of the distribution of $X_A$. Then by item 4, 
\[
\mathbb{P}(X_{A_k} \leq m+\delta) \to \mathbb{P}(X_A \leq m+\delta) < 1/2.
\]
This means that for large $k$, any median of $A_k$ must be $\geq m+\delta$, and so $m_k$ does not converge to $m$, a contradiction. The other case is handled similarly.

Using \eqref{eq: median_converge}, we now prove that $\mathcal{M} \cup \{0,1\}$ is closed. To show this, let $(m_k)$ be a sequence in $\mathcal{M} \setminus \{0,1\}$ that converges to some $m$. Then $m_k$ is a median of some $X_{A_k}$. If $(A_k)$ has a subsequence that converges to $\infty$, then $X_{A_k} \to 0$ in probability along this subsequence and, as in the proof of \eqref{eq: step_1_median_a}, the full sequence $m_k \to 0$. Similarly, if $(A_k)$ has a subsequence that converges to $-\infty$, then $m_k \to 1$. Therefore we need only consider the case where $(A_k)$ is bounded, and by passing to a subsequence, we may assume that $(A_k)$ converges to some finite $A$. Then by \eqref{eq: median_converge}, $m$ is a median of $X_A$, and so $m \in \mathcal{M}$. This means that $\mathcal{M} \cup \{0,1\}$ is closed.

Due to the results of the last paragraph, $\left( \mathcal{M} \cup \{0,1\}\right)^c$ is a countable union of disjoint open intervals, and so if we assume, for a contradiction, that $1/2 \notin \mathcal{M}$, it must be in one such interval $(a,b)$ with $0 < a < b < 1$ (by \eqref{eq: step_1_median_a} and \eqref{eq: step_1_median_b}) and $a,b \in \mathcal{M}$. Since $b \in \mathcal{M}$, it is in some $\mathcal{M}_A$. Let
\[
\hat A = \sup\{A : b \in \mathcal{M}_A\}.
\]
Note that $\hat A$ is finite because for $A$ large, all medians of $X_A$ are smaller than $1/4 < b$. Furthermore, $b$ must be a median of $X_{\hat A}$. Indeed, for each $k$, take $A_k$ such that $b \in \mathcal{M}_{A_k}$ and so that $A_k \to \hat A$. Choosing $m_k = b$ for all $k$ and applying \eqref{eq: median_converge}, we get that $b \in \mathcal{M}_{\hat A}$.

Now take a sequence $(A'_k)$ such that $A'_k \downarrow \hat A$ and $A'_k > \hat A$ for all $k$. Then since $b \notin \mathcal{M}_{A'_k}$ for all $k$, the monotonicity of item 4 implies that all medians of all $X_{A'_k}$'s must be $\leq a$. Next, choosing $m_k$ as a median of $X_{A'_k}$ and restricting to a subsequence so that the $m_k$'s converge to some $m$, we find from \eqref{eq: median_converge} that $m$ must be a median of $X_{\hat A}$. Since $\mathcal{M}_{\hat A}$ contains $a$ and $b$ and must be an interval, it contains $1/2$. This means $1/2 \in \mathcal{M}$, a contradiction.
\end{proof}

From this point forward, we fix a value of $A_n$ (and therefore of $B_n$) for which a median of $\mathbb{E}[T_n \mid \vec{N}]$ is equal to $(A_n+B_n)/2$. Next we need to define a set of ``good'' values of $\vec{N}$. These are values of $\vec{N}$ for which we can make the antichain argument of step 2 work. The definition of this good set $\mathcal{N}_n$ below will include two conditions. The first is to ensure that with high probability, $T(0,x_n)$ is well within the window of truncation defined in Lemma~\ref{lem: truncation}, so that $T(0,x_n) = T_n$. The second is a geodesic condition, used to ensure that decreasing $\vec{N}$ significantly (and thereby reducing the number of hi-mode edges) will decrease the conditional expectation $\mathbb{E}[T_n \mid \vec{N}]$.

We define
\begin{align*}
M_n &= \frac{A_n+B_n}{2}, \\
I_n & = [A_n, B_n] = \left[ M_n - \frac{\sqrt{\log \|x_n\|}}{2}, M_n + \frac{\sqrt{\log\|x_n\|}}{2}\right],
\end{align*}
and
\[
I_n' = \left[ M_n - \frac{\sqrt{\log \|x_n\|}}{4}, M_n + \frac{\sqrt{\log\|x_n\|}}{4} \right].
\]
Recall that a geodesic from $x$ to $y$ is a minimizing path in the definition of $T(x,y)$. It is known \cite[Theorem~4.2]{ADH_book} that in two dimensions, a.s.~a geodesic exists between any given $x$ and $y$. Last, recall that $K>0$ is the number defining $A(j)$ in \eqref{eq: annulus_def}.

\begin{defin}\label{def: good}
For a sequence $(\xi_n)$ of reals converging to 0, $\mathcal{N}_n$ is the set of those vectors $\vec{N}$ such that the following two conditions hold.
\begin{enumerate}
\item $\mathbb{P}\left(T(0,x_n) \in I_n' \mid \vec{N} \right) > 1-\xi_n$.
\item Let $E_n(j)$ be the event that every geodesic from $0$ to $x_n$ has at least $K^{j-1}$ hi-mode edges in $A(j)$. For every $j \in \left[0.25 \log_K \|x_n\|_\infty, 0.75 \log_K \|x_n\|_\infty\right]$, 
\[
\mathbb{P}\left( E_n(j) \mid \vec{N} \right) > 1-\xi_n.
\]
\end{enumerate}
\end{defin}

The main proposition about the good values of $\vec{N}$ is as follows.
\begin{prop}\label{prop: good}
There is a $K$ sufficiently large and a sequence of reals $(\xi_n)$ with $\xi_n \to 0$ such that
\[
\mathbb{P}(\vec{N} \in \mathcal{N}_n) > 1-\xi_n.
\]
\end{prop}
\begin{proof}
The proof of this proposition relies entirely on two lemmas.
\begin{lem}\label{lem: good_lemma_1}
One has
\[
\mathbb{P}(T(0,x_n) \in I_n') \to 1 \text{ as } n \to \infty.
\]
\end{lem}
\begin{proof}
Assume for a contradiction that $\mathbb{P}(T(0,x_n) \in I_n')$ does not converge to 1. So there is $\epsilon>0$ such that for some subsequence $(n_k)$, one has $\mathbb{P}(T(0,x_{n_k}) \notin I_{n_k}') > \epsilon$ for all $k$. This implies that for infinitely many $k$,
\begin{equation}\label{eq: good_possibility_1}
\mathbb{P}\left( \frac{T_{n_k} - A_{n_k}}{\sqrt{\log \|x_{n_k}\|}} < \frac{1}{4} \right) > \frac{\epsilon}{2},
\end{equation}
or $\mathbb{P}\left( (T_{n_k} - A_{n_k})/\sqrt{\log \|x_{n_k}\|} > 3/4 \right) > \epsilon/2$. Both cases are similar, so by restricting to a further subsequence $(n_k)$, we will assume that \eqref{eq: good_possibility_1} holds for all $k$.

Assumption~\ref{ass: main_assumption} implies that
\begin{equation}\label{eq: good_probability_convergence}
\frac{T(0,x_n) - a_n}{\sqrt{\log \|x_n\|}} \to 0 \text{ in probability}.
\end{equation}
We then write
\[
\frac{T_n-A_n}{\sqrt{\log \|x_n\|}} = \left( \left( \frac{T(0,x_n)-A_n}{\sqrt{\log \|x_n\|}} \right) \vee 0\right) \wedge 1 = \left( \left( \frac{a_n-A_n}{\sqrt{\log \|x_n\|}} + \frac{T(0,x_n)-a_n}{\sqrt{\log \|x_n\|}} \right) \vee 0 \right) \wedge 1.
\]
Due to \eqref{eq: good_probability_convergence}, if we choose a further subsequence $(n_k)$ so that 
\[
\left( \left( \frac{a_{n_k}-A_{n_k}}{\sqrt{\log \|x_{n_k}\|}}\right) \vee 0 \right) \wedge 1 \to \hat a \in [0,1],
\]
then
\[
\frac{T_{n_k}-A_{n_k}}{\sqrt{\log \|x_{n_k}\|}} \to \hat a \text{ in probability}.
\]
However, by \eqref{eq: good_possibility_1}, we obtain $\hat a \in [0,1/4]$. This means that
\[
\mathbb{P}\left( \frac{T_{n_k} - A_{n_k}}{\sqrt{\log \|x_{n_k}\|}} \leq \frac{1}{3} \right) \to 1,
\]
and this is a contradiction, since $1/2$ is a median of $(T_{n_k}-A_{n_k})/\sqrt{\log \|x_{n_k}\|}$.
\end{proof}

\begin{lem}\label{lem: good_lemma_2}
Writing $E_n$ for the intersection 
\[
E_n = \bigcap_{j \in [0.25 \log_K \|x_n\|_\infty, 0.75 \log_K \|x_n\|_\infty]} E_n(j),
\]
one has for some $K$ sufficiently large,
\[
\mathbb{P}(E_n)  \to 1 \text{ as } n \to \infty.
\]
\end{lem}
\begin{proof}
This result is not so hard to prove under stricter conditions on the edge-weights. Because we are making no moment assumption on the $t_e$'s, we must use more involved constructions of Cerf and Th\'eret \cite{CT}. Let $M$ be any number such that $\mathbb{P}(t_e\leq M) > p_c=1/2$ and say that an edge is open if it has weight $\leq M$ (closed otherwise). An (open) cluster is a maximal set of vertices such that any two vertices in the set are connected by a path whose edges are open. There is exactly one infinite cluster \cite[Theorem~8.1]{grimmett}, and so for each vertex $x$, we select a vertex $\tilde x$ to be the one of minimal distance from $x$ in the infinite cluster. If there is more than one candidate for $\tilde x$, then we break the tie in some deterministic way. Then, following \cite[Eq.~(1)]{CT}, we define $\tilde T(x,y) = T(\tilde x, \tilde y)$ for $x,y \in \mathbb{Z}^2$. Defining the time constant as
\[
\tilde g(x) = \lim_{n \to \infty} \frac{\tilde T(0,nx)}{n} \text{ a.s.~and in }L^1 \text{ for } x \in \mathbb{Z}^2,
\]
which \cite[Theorem~1]{CT} says exists, the statement of \cite[Theorem 3(i)]{CT} is a sort of shape theorem for $\tilde T$:
\begin{equation}\label{eq: weak_shape}
\lim_{n \to \infty} \sup_{x \in \mathbb{Z}^2,~\|x\|_1 \geq n} \left| \frac{\tilde T(0,x) - \tilde g(x)}{\|x\|_1} \right| = 0 \text{ a.s.}
\end{equation}
Last, \cite[Theorem~4]{CT} states that the limit $\tilde g$ is the same as that for $T$. In other words, if we define
\begin{equation}\label{eq: g_def}
g(x) = \lim_{n \to \infty} \frac{T(0,nx)}{n} \text{ in probability for } x \in \mathbb{Z}^2,
\end{equation}
then $g(x) = \tilde g(x)$ for all $x$.

Given these preparations, we first aim to show that there exists $\delta>0$ such that, if $\tilde F_\delta(x)$ is the event that every geodesic from $\tilde 0$ to $\tilde x$ contains at least $\delta\|x\|$ hi-mode edges, then
\begin{equation}\label{eq: geodesic_step_1}
\mathbb{P}\left( \cap_{\|x\| \geq R} \tilde F_\delta(x) \text{ occurs for }R \text{ large enough} \right) = 1.
\end{equation}
(The geodesic referred to is the $T$-geodesic between $\tilde 0$ and $\tilde x$.) To do this, we define an auxiliary set $(t_e')$ of edge weights:
\[
t_e' = \begin{cases}
t_e + 1 & \quad \text{if } t_e \text{ is hi-mode} \\
t_e & \quad \text{otherwise}.
\end{cases}
\]
Then $(t_e')$ stochastically dominates $(t_e)$ and, in particular, satisfies the concave ordering condition: for every concave increasing function $\Phi : \mathbb{R} \to \mathbb{R}$ and any edge $e$,
\[
\mathbb{E}\Phi(t_e) \leq \mathbb{E}\Phi(t_e'),
\]
so long as these expectations exist. So by \cite[Theorem~1.5]{Marchand}, one has $g(x) < g'(x)$ for all $x \neq 0$, where $g'$ is the limit defined as in \eqref{eq: g_def}, but for $t_e'$ instead of $t_e$. Because (see \cite[Sec.~2.3]{ADH_book}) both $g$ and $g'$ are restrictions of norms to $\mathbb{Z}^2$ (the condition \eqref{eq: perc_assumption} implies that neither are degenerate --- see \cite[Theorem~2.5]{ADH_book}), there exists $\eta>0$ such that
\begin{equation}\label{eq: bigger}
g'(x) \geq g(x) + \eta \|x\| \text{ for all } x \in \mathbb{Z}^2.
\end{equation}

Now, applying \eqref{eq: weak_shape} to both $T$ and $T'$, we obtain a random (a.s.~finite) number $R$ such that if $\|x\| \geq R$ then
\[
T(\tilde 0, \tilde x) < g(x) + \frac{\eta}{3} \|x\|,
\]
and
\[
T'(\tilde 0, \tilde x) > g'(x) - \frac{\eta}{3} \|x\|.
\]
(In applying \eqref{eq: weak_shape}, we need to choose two values of $M$: one for $T$ and one for $T'$ to construct the open edges for both sets of weights. We choose these values so that the open edges for both sets of weights are the same. Specifically, we can take $M> d_0$ (from \eqref{eq: d_0_def}) for $(t_e)$ and $M+1$ for $(t_e')$.) If we let $N(x,y)$ be the minimal number of hi-mode edges on any $T$-geodesic from $\tilde x$ to $\tilde y$, then for the $R$ above and $\|x\| \geq R$,
\[
g'(x) - \frac{\eta}{3} \|x\| < T'(\tilde 0, \tilde x) \leq T(\tilde 0 , \tilde x) + N(0,x) < g(x) + \frac{\eta}{3} \|x\| + N(0,x),
\]
or, using \eqref{eq: bigger},
\[
N(0,x) \geq g'(x) - g(x) - \frac{2\eta}{3}\|x\| \geq \frac{\eta}{3} \|x\|.
\]
This proves \eqref{eq: geodesic_step_1} with $\delta = \eta/3$.

Next we use \eqref{eq: geodesic_step_1} to prove a similar statement for points without tildes. That is, if we define $F_\delta(x)$ to be the event that every geodesic from $0$ to $x$ contains at least $\delta\|x\|$ number of hi-mode edges, then there exists $\delta>0$ such that
\begin{equation}\label{eq: geodesic_step_2}
\mathbb{P}\left( \cap_{\|x\| \geq R} F_\delta(x) \text{ occurs for }R \text{ large enough} \right) = 1.
\end{equation}
To prove this, we will show that for any $\eta>0$,
\begin{equation}\label{eq: geodesic_step_2_a}
\mathbb{P}\left( \cap_{\|x\| \geq R} F_\delta(x) \text{ occurs for } R \text{ large enough} \right) > 1-\eta.
\end{equation}
The main tool to prove this is a bound on the size of ``holes'' in the infinite cluster. Although this is particularly easy to do in two dimensions, it would lead us into bond percolation arguments, so instead we use a result of Kesten that would be applicable in any dimension. The following is a slight modification of \cite[Theorem~2.24]{aspects} and is given as \cite[Lemma~6.3]{ADH_gamma}. If $M$ is sufficiently large (so that $\mathbb{P}(t_e \leq M)$ is sufficiently close to 1), there exists $c_3>0$ such that for every $n$,
\[
\mathbb{P}\left(\text{each path from }0 \text{ to } \partial [-n,n]^2 \text{ intersects the infinite cluster}\right) > 1-e^{-c_3n}.
\]
Here, $\partial [-n,n]^2$ is the set of vertices of $[-n,n]^2$ with a neighbor outside of $[-n,n]^2$. So choose $n_0$ such that if $G_n$ is the event in this inequality, then
\begin{equation}\label{eq: geodesic_piece_1}
\mathbb{P}(G_{n_0}) > 1-\eta/3.
\end{equation}
Next, by translation invariance, the above estimate also implies
\begin{align*}
\mathbb{P}\bigg( \text{each path from } x \text{ to } x+ \partial [-\|x\|/2,\|x\|/2]^2 \text{ intersects the }&\text{infinite cluster}\bigg) \\
&> 1-e^{-c_3 \frac{\|x\|}{2}}.
\end{align*}
Therefore for all sufficiently large $R$,
\begin{align}
\mathbb{P}\left( \begin{array}{c}\text{for all }x \text{ with } \|x\| \geq R, \text{ each path from } x \text{ to } \\ x+ \partial [-\|x\|/2,\|x\|/2]^2 \text{ intersects the infinite cluster} \end{array} \right) &> 1- \sum_{\|x\| \geq R} e^{-c_3\frac{\|x\|}{2}} \nonumber \\
&> 1-\frac{\eta}{3}. \label{eq: geodesic_piece_2}
\end{align}

By translating the event in \eqref{eq: geodesic_step_1}, for each fixed $y \in [-n_0,n_0]^2$,
\[
\mathbb{P}\left( \cap_{\|x\| \geq R} \tilde F_{\delta/2}(y,x) \text{ occurs for all large }R\right) = 1,
\]
where $\tilde F_\delta(y,x)$ is the event that every geodesic from $\tilde y$ to $\tilde x$ contains at least $\delta \|x\|/2$ hi-mode edges. (Translating the event gives at least $\delta \|x-y\|$ hi-mode edges, but if $R$ is large and $y \in [-n_0,n_0]^2$, this number is larger than $\delta\|x\|/2$.) Therefore for all $R$ sufficiently large (depending on $n_0$),
\begin{equation}\label{eq: geodesic_piece_3}
\mathbb{P}\left( \cap_{y \in [-n_0,n_0]^2} \cap_{\|x\| \geq R} \tilde F_{\delta/2}(y,x) \right) > 1-\frac{\eta}{3}.
\end{equation}

Now suppose that the events in \eqref{eq: geodesic_piece_1}, \eqref{eq: geodesic_piece_2}, and \eqref{eq: geodesic_piece_3} occur (this has probability $>1-\eta$); we will show that for any $x$ with $\|x\| \geq 2R$ (assuming $R$ is sufficiently large), every geodesic from $0$ to $x$ contains at least $\delta\|x\|/4$ many hi-mode edges. So let $\Gamma$ be a geodesic from $0$ to such an $x$. Since the event in \eqref{eq: geodesic_piece_1} occurs, $\Gamma$ must contain a vertex $w_0$ of the infinite cluster in $[-n_0,n_0]^2$. Because the event in \eqref{eq: geodesic_piece_2} occurs, it must also contain a vertex $w_1$ of the infinite cluster in $x+[-\|x\|/2,\|x\|/2]^2$. Because the event in \eqref{eq: geodesic_piece_3} occurs, the segment of $\Gamma$ from $w_0$ to $w_1$ (which is a geodesic) must contain at least $\delta \|w_1\|/2 \geq \delta \|x\|/4$ many hi-mode edges. (Here we use that $\tilde w_0 = w_0$, $\tilde w_1 = w_1$, and $\|w_1\| \geq \|x\|/2 \geq R$.) This completes the proof of \eqref{eq: geodesic_step_2_a} and therefore of \eqref{eq: geodesic_step_2} (with $\delta/4$ in place of $\delta$).

Finally, we use \eqref{eq: geodesic_step_2} to show that if $K$ is large enough, then the statement of the lemma holds: $\mathbb{P}(E_n) \to 1$ as $n \to \infty$. For this, we need to use a geodesic length estimate from \cite[Theorem~4.9]{ADH_book}. It states that under assumption \eqref{eq: perc_assumption}, there exist $C_5,c_4>0$ such that
\[
\mathbb{P}\left( \text{each geodesic from } 0 \text{ to }x \text{ has at most }C_5\|x\| \text{ many edges}\right) \geq 1-e^{-c_4\|x\|^{1/2}}.
\]
So, given $\eta>0$, we can pick $R>0$ such that
\begin{equation}\label{eq: last_event_1}
\mathbb{P}\left( \cap_{\|x\| \geq R} \{\text{each geodesic from }0 \text{ to }x \text{ has at most }C_5\|x\| \text{ many edges}\}\right) > 1-\frac{\eta}{2}.
\end{equation}
By \eqref{eq: geodesic_step_2}, we can further increase $R$ so that
\begin{equation}\label{eq: last_event_2}
\mathbb{P}\left( \cap_{\|x\| \geq R} F_\delta(x)\right) > 1-\frac{\eta}{2}.
\end{equation}
Now pick $n$ so large that for all $j \in [0.25\log_K \|x_n\|_\infty-1, 0.75 \log_K\|x_n\|_\infty]$, each endpoint of an edge in $A(j)$ has Euclidean norm at least $R$. Suppose that the above two events occur (this has probability $>1-\eta$); we will show that $E_n$ occurs. Let $\Gamma$ be a geodesic from $0$ to $x_n$, let $w_0$ be the last endpoint of an edge in $A(j-1)$ that $\Gamma$ touches, and let $w_1$ be the vertex on $\Gamma$ directly before the first endpoint of an edge in $A(j+1)$ that it touches. Then the segment $\Gamma_0$ of $\Gamma$ from $0$ to $w_0$ is a geodesic, and the segment $\Gamma_1$ of $\Gamma$ from $w_0$ to $w_1$ is a geodesic whose edges are all in $A(j)$. Because $\|w_1\| \geq R$ and the event in \eqref{eq: last_event_2} occurs, the concatenation of $\Gamma_0$ and $\Gamma_1$ contains at least $\delta\|w_1\|$ many hi-mode edges. Because $\|w_0\| \geq R$ and the event in \eqref{eq: last_event_1} occurs, $\Gamma_0$ contains at most $C_5\|w_0\|$ many edges. Therefore
\[
\Gamma_1 \text{ contains at least } \delta \|w_1\| - C_5 \|w_0\| \text{ many hi-mode edges}.
\]
Because $w_1$ is adjacent to $A(j+1)$, one has $\|w_1\| \geq \|w_1\|_\infty = K^j-1$. Because $w_0$ is in $A(j-1)$, one has $\|w_0\| \leq \sqrt{2} \|w_0\|_\infty \leq \sqrt{2} K^{j-1}$. Therefore $\Gamma_1$ contains at least $\delta (K^j-1) - \sqrt{2} C_5 K^{j-1}$ many hi-mode edges. If $K$ is sufficiently large (note that $\delta$ is fixed), this number is $\geq K^{j-1}$. All these edges lie in $A(j)$ and are on $\Gamma$, so this completes the proof.
\end{proof}

Given Lemmas~\ref{lem: good_lemma_1} and \ref{lem: good_lemma_2}, we can complete the proof of Proposition~\ref{prop: good}. Taken together, they imply that we can find a sequence $(\epsilon_n)$ with $\epsilon_n \to 0$ such that $\mathbb{P}(E_n,~T(0,x_n) \in I_n') > 1-\epsilon_n.$ Now set $\xi_n = \sqrt{\epsilon_n}$ and define $U_n$, a function of $\vec{N}$, as
\[
U_n(\vec{N}) = \mathbb{P}\left( E_n,~T(0,x_n) \in I_n' \mid \vec{N} \right).
\]
Then
\begin{align*}
1- \xi_n^2 < \mathbb{P}(E_n,~ T(0,x_n) \in I_n') = \mathbb{E} U_n &= \mathbb{E}U_n \mathbf{1}_{\{U_n \leq 1-\xi_n\}} + \mathbb{E}U_n \mathbf{1}_{\{U_n > 1-\xi_n\}} \\
&\leq (1-\xi_n)\mathbb{P}(U_n \leq 1-\xi_n) + \mathbb{P}(U_n > 1-\xi_n).
\end{align*}
Rearranging, we obtain $\mathbb{P}(U_n \leq 1-\xi_n) < \xi_n$. This means that the set of $\vec{N}$ such that the probability defining $U_n$ is $>1-\xi_n$ has probability $>1-\xi_n$. For each $\vec{N}$ in this set, item 1 of Definition~\ref{def: good} holds. For item 2, the probability, for a fixed $j$, is no smaller than the probability of the event intersected over all relevant $j$'s. By definition of $E_n$, this probability is also $>1-\xi_n$, and this completes the proof.
\end{proof}

\subsection{Step 2. Small ball for conditional mean.}\label{sec: step_2}

In this step, we show that that the conditional expectation $\mathbb{E}[T_n \mid \vec{N}]$ satisfies a ``small ball'' probability bound: the probability that it lies in a small interval is bounded by $(\log \|x_n\|)^{-1/2}$.

\begin{prop}\label{prop: small_ball}
There exist $\epsilon, C_1 >0$ such that
\[
\sup_{r \in I_n'} \mathbb{P}\left( \mathbb{E}[T_n \mid \vec{N}] \in [r,r+\epsilon],~ \vec{N} \in \mathcal{N}_n \right) \leq \frac{C_1}{\sqrt{\log \|x_n\|}}.
\]
\end{prop}
\begin{proof}
At the risk of increasing the constant $C_1$, we will assume in the proof that $n$ is large.

The idea of the proof, inspired by that of \cite[Theorem~3.1]{AGKW}, is to represent each $N_j$ in the vector $\vec{N}$ as
\begin{equation}\label{eq: decomposition}
N_j \stackrel{d}{=} X_j + \eta_j Y_j,
\end{equation}
where $X_j \geq 0$ and $Y_j \geq 0$ are independent of $\eta_j$, which is a Bernoulli$(1/2)$ variable; that is, $\mathbb{P}(\eta_j = 0) = 1/2 = \mathbb{P}(\eta_j=1)=1/2$. Once we condition on all of the values of the $X_j$'s and $Y_j$'s, we view $\mathbb{E}[T_n \mid \vec{N}]$ as only a function of the $\eta_j$'s, and show that flipping one of the $\eta_j$'s from $1$ to $0$ (so that $N_j$ ``jumps down'' by the value $Y_j$) often decreases $\mathbb{E}[T_n \mid \vec{N}]$ by at least some $\epsilon$. This will be sufficient to show that the set whose probability we consider in the proposition is an ``antichain,'' and antichain probability estimates will complete the proof.

To find this representation, we let $F_j$ be the distribution function of $N_j$, and let $U_j$ be a uniform random variable on $[0,1]$. Then $F_j^{-1}(U_j)$ is distributed as $N_j$ (here $F_j^{-1}(t) = \inf\{x : F_j(x) \geq t\}$ is a generalized inverse of $F_j$). Let $X_j$ be a random variable with the same distribution as $N_j$ conditioned on $\{U_j \leq 1/2\}$ and $Z_j$ an independent random variable with the same distribution as $N_j$ conditioned on $\{U_j \geq 1/2\}$. Last, let $\eta_j$ be an independent Bernoulli$(1/2)$ random variable. Then for $Y_j = Z_j - X_j$, \eqref{eq: decomposition} holds. We will assume that the collection $\{X_1, X_2, \dots, Z_1, Z_2, \dots, \eta_1, \eta_2, \dots\}$ is formed of mutually independent random variables.

From now on, we think of $\mathbb{E}[T_n \mid \vec{N}]$ as a function of the vector $(X_1 + \eta_1Y_1, X_2 + \eta_2Y_2, \dots)$. We will want to consider only indices $j$ in which the corresponding jump value $Y_j$ is large enough (to ensure it affects the passage time). We also want to ensure $j$ is not too large or small, to avoid degeneracies. Therefore we define the set of indices
\[
\mathcal{I}_n = \left\{j \in \left[0.25 \log_K\|x_n\|_\infty, 0.75 \log_K\|x_n\|_\infty\right] : X_j \leq \mu_j - \sigma_j, Z_j \geq \mu_j + \sigma_j\right\}.
\]
Here, $\mu_j = \mathbb{E}N_j$, $\sigma_j = \sqrt{\mathrm{Var}~N_j}$, and we recall again that $K$ is the number defining the annuli $A(j)$ in \eqref{eq: annulus_def}. Then for $r \in I_n'$,
\begin{align}
\mathbb{P}\bigg( \mathbb{E}&[T_n \mid \vec{N}] \in [r,r+\epsilon],~\vec{N} \in \mathcal{N}_n\bigg) \nonumber \\
&= \mathbb{E} \left[ \mathbb{P}\left( \mathbb{E}[T_n \mid \vec{N}] \in [r,r+\epsilon],~\vec{N} \in \mathcal{N}_n \mid (X_i,Z_i)_{i \geq 1}, (\eta_i)_{i \notin \mathcal{I}_n}\right)\right]. \label{eq: to_pick_up}
\end{align}
Here we have conditioned on all the values of $X_i$ and $Z_i$ (which themselves determine the set $\mathcal{I}_n$), and also on the values of the Bernoullis $\eta$ outside $\mathcal{I}_n$. By independence, in the inner conditional probability, we can now view $\mathbb{E}[T_n \mid \vec{N}]$ as a function of only $(\eta_i)_{i \in \mathcal{I}_n}$ (with all other variables fixed).

Next we need the notion of an antichain. Here, for convenience, it is defined using the ordering $(\leq)$, which is opposite of what is normally used $(\geq$). 
\begin{defin}
A subset $\Xi$ of $\{0,1\}^n$ is an antichain if whenever $\eta = (\eta_1, \ldots, \eta_n) \in \Xi$ and $\tau \neq \eta$ satisfies $\tau_i \leq \eta_i$ for all $i$, then $\tau \notin \Xi$.
\end{defin}
A simple example of an antichain is the set $\{\eta : \sum_i \eta_i = k\}$. The asymptotic bound $n^{-1/2}$ on the probability of this set under the uniform measure extends to all antichains. Indeed, by Sperner's Theorem \cite{sperner}, an antichain cannot contain more than $\binom{n}{\lfloor n/2\rfloor}$ many elements. Therefore, under the uniform measure, 
\begin{equation}\label{eq: antichain_bound}
\mathbb{P}(\Xi) \leq \frac{8}{\sqrt{n}} \text{ for any antichain } \Xi \subset \{0,1\}^n.
\end{equation}

To apply the antichain bound, we must establish the following.
\begin{lem}\label{lem:showanti}
There exist $\epsilon>0$ and $K$ sufficiently large such that for all large $n$, all $r \in I_n'$, and any choice of fixed $(X_i,Z_i)_{i \geq 1}$ and $(\eta_i)_{i \notin \mathcal{I}_n}$, the set
\[
Q = Q(n, (X_i, Z_i)_{i \geq 1}, r, (\eta_i)_{i \notin \mathcal{I}_n}) = \left\{ (\eta_i)_{i \in \mathcal{I}_n} : \mathbb{E}[T_n \mid \vec{N}]((\eta_i)) \in [r,r+\epsilon],~\vec{N} \in \mathcal{N}_n \right\}
\]
is an antichain in $\{0,1\}^{\#\mathcal{I}_n}$.
\end{lem}

Given $x, y$ and a realization $(t_e)$ of the edge weights such that a geodesic from $x$ to $y$ exists, we write $\geo(x,y) = \geo[(t_e)](x,y)$ for the union of all edges appearing in geodesics from $x$ to $y$.  For use in the proof of Lemma \ref{lem:showanti}, we note the following fact, whose proof we omit.
\begin{prop}\label{prop:linear}
	Suppose $(t_e)$ and $(t_e')$ are two edge-weight configurations such that the following hold: i) $t_f' = t_f - \varepsilon$, ii) $t'_e \leq t_e$ for $e \neq f$, iii) a geodesic from $x$ to $y$ exists in the configuration $(t_e)$, and iv) $f \in \geo[(t_e)](x,y)$. Then $T(x,y)[(t_e')] \leq T(x,y)[(t_e)] - \varepsilon$.
\end{prop}

\begin{proof}[Proof of Lemma \ref{lem:showanti}]
	Fix values of the arguments of $Q$ as above. Let $\eta^{(1)} = (\eta_i^{(1)})$ and $\eta^{(2)} = (\eta_i^{(2)})$
	 be realizations of the $\eta$ variables having 
	 $(\eta_i^{(1)})_{i \notin \mathcal{I}_n} = (\eta_i^{(2)})_{i \notin \mathcal{I}_n} =(\eta_i)_{i \notin \mathcal{I}_n}$,
	 such that $\mathbb{E}[T_n \mid \vec{N}]( \eta^{(1)}) \in [r, r+ \epsilon]$ and such that $\vec{N}(\eta^{(1)}) \in \mathcal{N}_n$. In other words, $(\eta_i^{(1)})_{i \in \mathcal{I}_n} \in Q$. Suppose $\eta^{(2)} \neq \eta^{(1)}$ but that $\eta^{(2)} \leq \eta^{(1)}.$ We will show $(\eta_i^{(2)})_{i \in \mathcal{I}_n} \notin Q$ by showing $\pe[T_n \mid \vec{N}](\eta^{(2)}) < r$.
	 
	 In fact, it suffices to show the preceding statement for $\eta^{(2)}$ that differs from $\eta^{(1)}$ in only one index $j \in \mathcal{I}_n$, where $\eta^{(2)}_j = 0$ and $\eta^{(1)}_j = 1$. Indeed, recall that $(t_e)$ is monotone in $\vec{N}$, and $\vec{N}$ is monotone in $\eta$. This means that, if we prove $\pe[T_n \mid \vec{N}](\widetilde \eta^{(2)}) < r$ for some arbitrarily chosen $\widetilde \eta^{(2)} \geq \eta^{(2)}$ differing from $\eta^{(1)}$ in only one index as above, then we necessarily have $\pe[T_n \mid \vec{N}](\eta^{(2)}) < r$ as well. Thus, we will spend the rest of the proof analyzing this simpler case for a fixed arbitrary $j \in \mathcal{I}_n$.
	 
	 We write the difference of the conditional expectations in the two configurations as
	 \begin{align}\label{eq:intout}
	 	\mathbb{E}[T_n \mid \vec{N}](\eta^{(1)}) - \mathbb{E}[T_n \mid \vec{N}](\eta^{(2)}) = \int\left( T_n(\vec N^{(1)}, \Pi, P) - T_n(\vec N^{(2)}, \Pi, P)\right) \, \mathrm{d}(\Pi, P).
	 \end{align}
	 Here $\vec{N}^{(\ell)} = \vec{N}^{(\ell)}((\eta^{(\ell)}_i)_{i \in \mathcal{I}})$ is $\vec{N}$ in the configuration $((X_i, Z_i)_{i \geq 1}, (\eta_i^{(\ell)})_{i \geq 1})$ but viewed as a function only of the $\eta$ coordinates in $\mathcal{I}_n$. Our next step is defining a special set $\Upsilon$ of pairs $(\Pi,P)$ which will be useful for lower-bounding this integral.
	 
	 Write $\geo^{(1)}(0,x_n)$ for the union of geodesics from $0$ to $x_n$ in $(\vec{N}(\eta^{(1)}), \Pi, P)$.
	 We define 
	 \begin{align*}
	 \Upsilon &= \Upsilon(n, (X_i, Z_i)_{i \geq 1}, (\eta_i)_{i \notin \mathcal{I}_n}, (\eta^{(1)}_i)_{i \in \mathcal{I}_n},j) \\
	 &= \left\{ (\Pi,P) : \text{ for some }f \in A(j) \cap \geo^{(1)}(0, x_n), \text{ we have }\pi_j(f) \in (N^{(2)}_j, N^{(1)}_j]\right\}.
	 \end{align*}
	 Note that the interval $(N^{(2)}_j, N^{(1)}_j]$ has length at least $2\sigma_j$ by the very definition of $\mathcal{I}_n$. The definition of $\Upsilon$ ensures that the value of $t_f$ changes from hi-mode to lo-mode when $\eta^{(1)}$ is replaced by $\eta^{(2)}$ (i.e., when $\eta^{(1)}_j$ is ``flipped'').
	 
	 
	 We will need to be able to make a measurable choice of edge $f$ as in the definition of $\Upsilon$ when lower-bounding the integral in \eqref{eq:intout}. On the event $\Upsilon$, we define $e(0)$ to be the edge $f \in A(j) \cap \geo^{(1)}(0, x_n)$, which satisfies $\pi_j(f) \in (N^{(2)}_j, N^{(1)}_j]$ and which is smallest in some deterministic ordering. We will not use $e(0)$ on the event $\Upsilon^c$, so we can choose its value on that event arbitrarily. We note here the most important property of the definition of $e(0)$:
	 \begin{equation}
	 \label{eq:teprop}
	  \text{Conditonal on } \Upsilon \cap \{e(0) = f\}\, ,\  t_{f}^{(L)} \text{ is distributed as } F \text{ conditional on } t_f \leq d_0\ .
	 \end{equation}
	 
	The distributional claim \eqref{eq:teprop} follows from the fact that $\geo^{(1)}(0, x_n)$ depends only on $\vec{N}^{(1)}$, on $\Pi$, and (for fixed $\vec{N}^{(1)}$ and $\Pi$) on the values of those entries of $P$ which actually appear in the edge-weight configuration $(t_e)$ for the outcome $(\vec N^{(1)}, \Pi, P)$. In other words, on the following entries of $P$:
	\begin{align*}\bigg\{t_e^{(L)}:& \pi_k(e) > N^{(1)}_k \text{ for the } k \text{ such that } e \in A(k)\bigg\} \\ &\bigcup\left\{t_e^{(H)}:\,\pi_k(e)\leq N^{(1)}_k \text{ for the } k \text{ such that } e \in A(k) \right\}. \end{align*}
	The independence of the $(L)$- and $(H)$-coordinates of $P$ thus shows \eqref{eq:teprop}.
	
	We now are prepared to prove a probability lower bound for $\Upsilon$.
	\begin{clam}\label{clam:upslb} There exists $\delta_0 > 0$ such that, uniformly in $n,$ $(X_i, Z_i)_{i \geq 1}$, $r \in I'_n$, $(\eta_i)_{i \notin \mathcal{I}_n}$, and $(\eta^{(1)}_i)_{i \in \mathcal{I}_n} \in Q$, we have $\prob((\Pi,P) \in \Upsilon) \geq \delta_0$.
		\end{clam}
		\begin{proof}
			Fixing the arguments of $\Upsilon$ as above, we consider possible values of $\pi_j$. Define the set $U_2 = \{e \in A(j):\, \pi_j(e) \in [0, N_j^{(2)}]\}$, the set $U_1 = \{e \in A(j):\, \pi_j(e) \in (N_j^{(2)}, N_j^{(1)}]\},$ and $U_3 = A(j) \setminus (U_1 \cup U_2).$ We control $\prob((\Pi, P) \in \Upsilon^c)$ by arguing that a large number of $e \in A(j)$ must lie in $U_1 \cup U_2$, and then that, in fact, many of these $e$ must lie in $U_1$.
			
			Indeed, writing $\geo^{(1)} = \geo^{(1)}(0, x_n)$, we see
			\begin{align}\prob((\Pi,P) \in \Upsilon^c) &\leq \prob(\#(U_1 \cup U_2) \cap \geo^{(1)} < K^{j-1}) \nonumber \\
			&+   \prob(U_1 \cap \geo^{(1)} = \varnothing, \#(U_1 \cup U_2) \cap \geo^{(1)} \geq K^{j-1}).\label{eq:upsilonbd}
			\end{align}
			Note that both the events above only depend on $\Pi$ and $P$ (since $\vec{N} = \vec{N}^{(1)}$ is fixed). Also the event appearing in the first probability is equal to the event $\{\geo^{(1)} \cap A(j) \text{ contains } < K^{j-1} \text{ hi-mode edges}\}$. Since ``$\prob$'' in that term really means averaging over $(\Pi, P)$ for fixed $\vec{N} = \vec{N}^{(1)}$, this is a conditional probability (given $\vec{N} = \vec{N}^{(1)}$). By the second item of Definition~\ref{def: good} of the set $\mathcal{N}_n$ and the fact that $\vec{N}^{(1)} \in \mathcal{N}_n$, this probability is bounded above by $\xi_n$. 
			
			We turn to bounding the second term on the right-hand side of \eqref{eq:upsilonbd}. We bound this probability by conditioning on the value of $U_1 \cup U_2$ and $P$ as well as the permutations in $A(k)$ for $k \neq j$. In other words, we write
			\begin{align*}
			\prob(U_1 &\cap \geo^{(1)} = \varnothing, \#(U_1 \cup U_2) \cap \geo^{(1)} \geq K^{j-1})\\ = &\pe \left[ \prob\left(U_1 \cap \geo^{(1)}
			 = \varnothing, \#(U_1 \cup U_2) \cap \geo^{(1)} \geq K^{j-1} \bigm| U_1 \cup U_2 , P, (\pi_k)_{k \neq j} \right)  \right]\\
			 = &\pe \left[ \prob\left(U_1 \cap \geo^{(1)}
			 = \varnothing \bigm| U_1 \cup U_2 , P, (\pi_k)_{k \neq j} \right);\, \#(U_1 \cup U_2) \cap \geo^{(1)} \geq K^{j-1}  \right] .
			\end{align*}
			We have used the fact that (given our other variables, which determine $\vec{N}^{(1)}$ and $\vec{N}^{(2)}$) $\geo^{(1)}$ depends only on $P$ and on $U_1 \cup U_2$, but not directly on $U_1$ or $U_2$. This fact also ensures that $\geo^{(1)}$ is treated as constant when computing the above conditional probability. Moreover, the 
			conditional distribution of $U_1$ is just the distribution of the first $N_j^{(1)} - N_j^{(2)}$ elements of a uniform permutation of $U_1 \cup U_2$. Thus, the conditional probability above is the same as the probability that, under a uniform permutation $\sigma$ of $\{1, \ldots, N_j^{(1)}\}$, no element $k \in \{1, \ldots, \#((U_1 \cup U_2) \cap \geo^{(1)})\}$ has $\sigma(k) \in \{1, \ldots, N_j^{(1)} - N_j^{(2)}\}$.
			
			We can now invoke the bounds on $\#((U_1 \cup U_2) \cap \geo^{(1)})$ and $N_j^{(1)} - N_j^{(2)}$. By the very definition of $\mathcal{I}_n$, we have $N_j^{(1)} - N_j^{(2)} \geq 2\sigma_j \geq c_1 K^j$ for some $c_1 > 0$. Moreover, for each outcome considered in the expectation above, we have $\#((U_1 \cup U_2) \cap \geo^{(1)}) \geq K^{j-1}$, and we of course have the trivial upper bound $N_j^{(1)} \leq \#A(j) \leq 18K^{2j}$. So the probability involving $\sigma$ described in the preceding paragraph is at most
			\begin{equation}
			\label{eq:sigmaprob}
			 \left( \frac{N_j^{(1)} - \#((U_1 \cup U_2) \cap \geo^{(1)})}{N_j^{(1)}}\right)^{N_j^{(1)}- N_j^{(2)}} \leq  \left(1 - \frac{c}{ K^j}\right)^{c_1 K^j},
			\end{equation}
			uniformly in $j$. The right-hand side of \eqref{eq:sigmaprob} is less than one for each $j$, and converges as $j \to \infty$ to $e^{-c_1 c} < 1$. Thus, $\prob((\Pi,P) \in \Upsilon^c) < 1 - \delta_0$, for some $\delta_0 > 0$, uniformly in $j$ large.
			\end{proof}
			We will lower-bound the integral in \eqref{eq:intout} by restricting it to $\Upsilon$ and then integrating iteratively, first conditioning on the information needed to determine $\geo^{(1)}$. Claim \ref{clam:upslb} will help bound the integral restricted to $\Upsilon$. 
			
			Define the event
			\begin{equation}
			\label{eq:upsilon1def}
			\Upsilon_1 = \Upsilon_1(\vec{N}^{(1)}) =  \{(\Pi, P): T(0,x_n) \in I_n' \text{ in } (t_e^{(1)})\}, \end{equation}
			and write $\Upsilon_2 = \Upsilon \cap \Upsilon_1$. Extending the reasoning used to establish \eqref{eq:teprop}, we see that, conditional on $\Upsilon_2$ and $e(0)=f$, the value of $t_{f}^{(L)}$ has distribution $F$ conditional on $t_f \leq d_0$. We can then lower-bound \eqref{eq:intout} by
			\begin{align}\nonumber\int_{\Upsilon_2} \bigg[T_n&(\vec N^{(1)}, \Pi, P) - T_n(\vec N^{(2)}, \Pi, P)\bigg]  \, \mathrm{d}(\Pi,P)
			\\
			&=\label{eq:braket}\sum_{f}\int_{\Upsilon_2} \left[T_n(\vec N^{(1)}, \Pi, P) - T_n(\vec N^{(2)}, \Pi, P)\right]  \, \mathbf{1}_{\{e(0) = f\}} \ \mathrm{d}(\Pi,P),
			\end{align}
			where we have used monotonicity of the passage time in $\vec{N}$ (which ensures the integrand is a.s.~nonnegative). 
			
			We write $\widehat P_f$ for the collection of $t_e^{(L)}$ and $t_e^{(H)}$ for $e \neq f$, along with $t_f^{(H)}$. In other words, $\widehat P_f = (t_f^{(H)}) \oplus (t_e^{(L)}, t_e^{(H)})_{f \neq e \in \edge^2}$. Using Fubini's theorem, we can rewrite the integral in \eqref{eq:braket} as
			\begin{equation}\label{eq: liable}
			\sum_{f}\int_{\Upsilon_2} \left[ \int \left[T_n(\vec N^{(1)}, \Pi, P) - T_n(\vec N^{(2)}, \Pi, P)\right]  \,  \ \mathrm{d}t_f^{(L)}\ \right] \mathbf{1}_{\{e(0) = f\}} \mathrm{d}(\Pi, \widehat P_f), 
			\end{equation}
			using that the event $\{e(0) = f\} \cap \Upsilon_2$ is measurable with respect to~$(\Pi, \widehat P_f)$.
			
			\begin{clam}\label{clam:innersanctum}
				There exists a $c > 0$ such that uniformly in $n$ large and $f$, in given values of the parameters in the argument of $\Upsilon$, and in $\Pi$ and $\widehat P_f$ such that $\Upsilon_2 \cap \{e(0) = f\}$ occurs, the following holds:
				\begin{equation}\label{eq:pint}
				\int\left( T_n(\vec N^{(1)}, \Pi, P) - T_n(\vec N^{(2)}, \Pi, P)\right) \, \mathrm{d}t_f^{(L)} \geq c. \end{equation}
				\end{clam}
				\begin{proof}
					For $\vec{N}^{(1)},$ $\Pi$, and $\widehat P_f$ as in the claim, we note that 
					\[
					T_n(\vec{N}^{(1)}, \Pi, P) = T(0, x_n)[(\vec{N}^{(1)}, \Pi, P)] \text{ for a.e. }t_f^{(L)},
					\] 
					since $\pi_j(f) \in (N_j^{(2)}, N_j^{(1)}]$ and since $I_n' \subseteq I_n$. We lower-bound $T(0, x_n)[(\vec N^{(1)}, \Pi, P)] - T_n(\vec N^{(2)}, \Pi, P)$ by first lower-bounding $T(0, x_n)[(\vec N^{(1)}, \Pi, P)] - T(0, x_n)[(\vec N^{(2)}, \Pi, P)].$
					
					By Proposition \ref{prop:linear}, using the fact $\pi_j(f) \in (N_j^{(2)}, N_j^{(1)}]$, we have the following on $\{e(0) = f\} \cap \Upsilon_2$:
					\begin{align*}T(0,x_n)[(\vec N^{(1)}, \Pi, P)] - T(0,x_n)[(\vec N^{(2)}, \Pi, P)] &\geq  t_f^{(H)} - t_f^{(L)}\\
					&\geq d_0 - t_{f}^{(L)}.
					\end{align*} 
					We now return to the truncated variable $T_n$. Since $T(0,x_n)[(\vec N^{(1)}, \Pi, P)]\in I'_n$ and since $T_n$ is truncated below at $M_n - \sqrt{\log \|x_n\|}/2$, the above shows that
					\[T_n[(\vec N^{(1)}, \Pi, P)] - T_n(\vec N^{(2)}, \Pi, P) \geq  d_0 - t_{f}^{(L)},\]
					for all $n$ large enough that $\sqrt{\log \|x_n\|}/4 > d_0$.
					Integrating this inequality gives by \eqref{eq: implication_d_0}
					\[
					\text{LHS of } \eqref{eq:pint} \geq d_0 - \int t_f^{(L)}~\text{d}t_f^{(L)} = d_0 - \frac{\mathbb{E}t_e \mathbf{1}_{\{t_e \leq d_0\}}}{\mathbb{P}(t_e \leq d_0)} \geq c,
					\]
					for some $c>0$. This completes the proof.
				\end{proof}
				Having proved the two claims above, the result follows in short order. Recall our goal is to show $(\eta_i^{(2)})_{i \in \mathcal{I}_n} \notin Q$ by showing $\pe[T_n \mid \vec{N}](\eta^{(2)}) < r$, if $\varepsilon$ is small enough. By the result of Claim \ref{clam:innersanctum} placed into \eqref{eq: liable}, we have
				\[\pe[T_n \mid \vec{N}](\eta^{(2)}) \leq \pe[T_n \mid \vec{N}](\eta^{(1)}) - c \prob((\Pi,P) \in \Upsilon_2). \]
				All that remains is to show that $\prob(\Upsilon_2) > c_1 > 0$ uniformly.
				
				We have $\prob((\Pi, P) \in \Upsilon_2^c) \leq \prob((\Pi, P) \in \Upsilon^c) + \prob((\Pi, P) \in \Upsilon_1^c)$, and the first term is at most $1 - \delta_0 < 1$ by Claim \ref{clam:upslb}. The second is at most $\xi_n \to 0$ by the definition of $\mathcal{N}_n$ (Definition \ref{def: good}). So $\mathbb{P}((\Pi, P) \in \Upsilon_2) > \delta_0/2$ for $n$ large, and this completes the proof of Lemma~\ref{lem:showanti}.
\end{proof}

Returning to \eqref{eq: to_pick_up} (in an effort to complete the proof of Proposition~\ref{prop: small_ball}), and applying the antichain bound \eqref{eq: antichain_bound}, we obtain
\begin{equation}\label{eq: near_the_end}
\mathbb{P}\left( \mathbb{E}[T_n \mid \vec{N}] \in [r,r+\epsilon],~\vec{N} \in \mathcal{N}_n\right) \leq 8 \mathbb{E}\left[ \frac{1}{\sqrt{\#\mathcal{I}_n}} \wedge 1\right].
\end{equation}
To estimate this quantity, we write
\[
\#\mathcal{I}_n = \sum_{j= \lceil 0.25 \log_K \|x_n\|_\infty \rceil}^{\lfloor 0.75 \log_K\|x_n\|_\infty \rfloor} \mathbf{1}_{\{X_j \leq \mu_j - \sigma_j\}} \mathbf{1}_{\{Z_j \geq \mu_j + \sigma_j\}},
\]
where the collection of all summands is independent. They are Bernoulli random variables, so to estimate their parameters, we compute
\[
\mathbb{E}\mathbf{1}_{\{X_j \leq \mu_j - \sigma_j\}} = \frac{\mathbb{P}(X_j \leq \mu_j - \sigma_j , U_j \leq 1/2)}{\mathbb{P}(U_j \leq 1/2)} = 2 \mathbb{P}(F_j^{-1}(U_j) \leq \mu_j - \sigma_j, U_j \leq 1/2).
\]
Since $F_j^{-1}(U_j)$ is a Binomial random variable with parameters $\#A(j)$ and $1-F(d_0)$, the central limit theorem implies
\[
\lim_{n \to \infty} \sup_{0.25 \log_K \|x_n\| \leq j \leq 0.75 \log_K\|x_n\|} \left| \mathbb{P}(F_j^{-1}(U_j) \leq \mu_j-\sigma_j) - \frac{1}{\sqrt{2\pi}} \int_{-\infty}^{-1} e^{-x^2/2}~\text{d}x\right| = 0.
\]
Since this last integral is $<1/2$, when $n$ is large, every $j$ in the specified range has the property $\{F_j^{-1}(U_j) \leq \mu_j - \sigma_j\} \subset \{U_j \leq 1/2\}$. Using this in the above equation and since the integral is also $>0.1$ gives for large $n$,
\[
\mathbb{E} \mathbf{1}_{\{X_j \leq \mu_j - \sigma_j\}} = 2 \mathbb{P}(F_j^{-1}(U_j) \leq \mu_j - \sigma_j) \geq 0.2.
\]
A similar argument works to show that if $n$ is large,
\begin{equation*}\label{eq: bernoulli_estimate}
\mathbb{E}\mathbf{1}_{\{Z_j \geq \mu_j + \sigma_j\}} \geq 0.2.
\end{equation*}
Therefore if $(b_j)$ is a collection of i.i.d.~Bernoulli random variables with parameter $0.04$, if $n$ is large,
\begin{align}
\mathbb{P}(\#\mathcal{I}_n \leq 0.01  \log_K \|x_n\|_\infty) &\leq  \mathbb{P}\left( \sum_{j= \lceil 0.25 \log_K \|x_n\|_\infty \rceil}^{\lfloor 0.75 \log_K\|x_n\|_\infty \rfloor} b_j \leq 0.01  \log_K \|x_n\|_\infty\right) \nonumber \\
&\leq \exp\left( -c_1 \log_K \|x_n\|_\infty\right) \label{eq: houdres}
\end{align}
for some $c_1>0$. (Here we have used standard large deviation inequalities for sums of i.i.d.~Bernoullis.) Putting \eqref{eq: houdres} in \eqref{eq: near_the_end} gives for all $r \in I_n'$ and $n$ large
\begin{align*}
\mathbb{P}\left( \mathbb{E}[T_n \mid \vec{N}] \in [r,r+\epsilon],~\vec{N} \in \mathcal{N}_n\right) &\leq 8 \exp\left( - c_1 \log_K\|x_n\|_\infty \right) + \frac{8}{\sqrt{0.01 \log_K \|x_n\|_\infty}} \\
&\leq \frac{C_1}{\sqrt{\log \|x_n\|}}.
\end{align*}
This completes the proof.
\end{proof}

%

\subsection{Step 3. Reckoning.}\label{sec: step_3}

From the small ball result and the truncation, we can give a fluctuation result for $T(0,x_n)$. This will contradict Assumption~\ref{ass: main_assumption}.
\begin{prop}\label{prop: contradiction}
There exists $c_1>0$ such that for all large $n$,
\[
\mathbb{P}(T(0,x_n) \leq M_n - c_1\sqrt{\log \|x_n\|}) > c_1,
\]
and
\[
\mathbb{P}(T(0,x_n) \geq M_n + c_1\sqrt{\log \|x_n\|}) > c_1.
\]
\end{prop}
\begin{proof}
We first use the result of step 2 to show that there exists $c_2\in (0,1/4)$ such that for all large $n$,
\[
\mathbb{P}\bigg( \mathbb{E}[T_n \mid \vec{N}] \leq M_n - c_2 \sqrt{\log \|x_n\|}\bigg) > c_2,
\]
and
\begin{equation}\label{eq: first_reckoning}
\mathbb{P}\bigg( \mathbb{E}[T_n \mid \vec{N}] \geq M_n + c_2 \sqrt{\log \|x_n\|}\bigg) > c_2.
\end{equation}
We will show the second inequality; the first is similar. To do so, we cover the interval $\left[ M_n, M_n + c_2\sqrt{\log \|x_n\|}\right]$ by a collection $\mathfrak{I}$ of $\left\lceil c_2 \sqrt{\log \|x_n\|} / \epsilon\right\rceil$ many closed intervals of length $\epsilon$, all of which are contained in $I_n'$. (Here, $\epsilon$ is from Proposition~\ref{prop: small_ball}.) Then we upper bound
\begin{align*}
\mathbb{P}\left( \mathbb{E}[T_n \mid \vec{N}] \in \left[ M_n, M_n+c_2\sqrt{\log \|x_n\|} \right]\right) \leq \sum_{\mathcal{I} \in \mathfrak{I}} &\mathbb{P}(\mathbb{E}[T_n \mid \vec{N}] \in \mathcal{I}, ~\vec{N} \in \mathcal{N}_n) \\
+ &\mathbb{P}(\vec{N} \notin \mathcal{N}_n).
\end{align*}
By Proposition~\ref{prop: good} and Proposition~\ref{prop: small_ball}, the above is bounded by
\[
\frac{C_1}{\sqrt{\log \|x_n\|}} \left\lceil \frac{c_2}{\epsilon} \sqrt{\log \|x_n\|}\right\rceil +\xi_n.
\]
For $n$ large, $\xi_n < 1/8$. Also, if we choose $c_2 < \frac{\epsilon}{16C_1}$, then the first term is $<1/8$. Therefore for such choices, we have
\[
\mathbb{P}\left( \mathbb{E}[T_n \mid \vec{N}] \in \left[ M_n, M_n+c_2\sqrt{\log \|x_n\|}\right]\right) < \frac{1}{4}.
\]
Since $M_n$ is a median of $\mathbb{E}[T_n \mid \vec{N}]$, the left side of \eqref{eq: first_reckoning} is at least
\[
\frac{1}{2} - \mathbb{P}\left( \mathbb{E}[T_n \mid \vec{N}] \in \left[ M_n, M_n+c_2\sqrt{\log \|x_n\|} \right]\right) > \frac{1}{4}.
\]
This proves \eqref{eq: first_reckoning}.

It remains to show that the proposition follows from \eqref{eq: first_reckoning} and the fact that $T_n$ is defined using a truncation. Again, because the statements have similar proofs, we will only show the second. Now define
\[
\mathsf{E} = \{\vec{N} : \mathbb{E}[T_n \mid \vec{N}] \geq M_n + c_2\sqrt{\log \|x_n\|}\}
\]
and for a given $\vec{N}$,
\[
\mathsf{A} = \mathsf{A}(\vec{N})=  \left\{ (\Pi, P) : T_n(\vec{N},\Pi,P) \geq M_n + (c_2/2)\sqrt{\log \|x_n\|}\right\}.
\]
Then we compute
\begin{align*}
\mathbb{E}[T_n \mid \vec{N}] - M_n = \int (T_n - M_n) ~\text{d}(\Pi,P) &= \int(T_n-M_n) \mathbf{1}_{\mathsf{A}}(\Pi,P)~\text{d}(\Pi,P) \\
&+ \int(T_n-M_n) \mathbf{1}_{\mathsf{A}^c}(\Pi,P)~\text{d}(\Pi,P)  \\
&\leq \frac{\sqrt{\log \|x_n\|}}{2} \mathbb{P}((\Pi,P) \in \mathsf{A}) \\
&+ \frac{c_2}{2}\sqrt{\log \|x_n\|} \mathbb{P}((\Pi,P)\in \mathsf{A}^c).
\end{align*}
If $\vec{N} \in \mathsf{E}$, we obtain
\[
c_2 \leq \frac{1}{2} \mathbb{P}((\Pi,P) \in \mathsf{A}) + \frac{c_2}{2} \mathbb{P}((\Pi,P) \in \mathsf{A}^c),
\]
or
\begin{equation}\label{eq: near_end}
\mathbb{P}((\Pi,P) \in \mathsf{A}) \geq \frac{c_2}{1-c_2}.
\end{equation}
This implies
\begin{align*}
\mathbb{P}\left( T(0,x_n) \geq M_n + (c_2/2) \sqrt{\log \|x_n\|}\right) &= \mathbb{P}\left( T_n \geq M_n + (c_2/2) \sqrt{\log \|x_n\|}\right) \\
&= \int \mathbb{P}((\Pi,P) \in \mathsf{A})~\text{d}\vec{N} \\
&\geq \int_{\mathsf{E}} \mathbb{P}((\Pi,P) \in \mathsf{A})~\text{d}\vec{N}.
\end{align*}
By \eqref{eq: first_reckoning} and \eqref{eq: near_end}, this is bounded below by $c_2^2/(1-c_2)$. Taking $c_1$ small, therefore, completes the proof of the proposition.
\end{proof}

\section{Sketch of proof of Theorem~\ref{thm: cylinders}}\label{sec: cylinders}

The proof of Theorem~\ref{thm: cylinders} follows the same lines as those of the proof of Theorem~\ref{thm: main_thm}, so we only indicate here the changes needed. The main difference is that we do not need to consider changing weights from hi- to lo-mode in all annuli $A(k)$ as we did before. It is sufficient to focus on only one of the top-scale annuli, the annulus $A(k_0)$, with $k_0 = \lfloor \log_K \|x_n\|_\infty \rfloor$. 

Roughly speaking, the reason why one can obtain a polynomial lower bound for fluctuations of the passage time in a cylinder (whereas the bound in the full space is only logarithmic) is as follows. In the full space, a geodesic from $0$ to $x_n$ contains at least order $K^j$ many edges in the annulus $A(j)$. If we resample $N_j$ (number of hi-mode weights associated to this annulus), it has a positive probability, independent of $j$, to decrease by order $K^j$ as well, since this is the standard deviation of $N_j$. The edges where weights decrease are uniformly distributed in the annulus, so any given edge has probability $\sim K^j/K^{2j} = 1/K^j$ to be one. This means the expected number of these decreased edges which lie on the geodesic is at least order $K^j \cdot 1/K^j = 1$. In other words, when we decrease $N_j$, typically there will be one edge on the geodesic whose weight decreases, and so $T(0,x_n)$ will decrease by a constant. In a cylinder, the counting is more favorable. Still the geodesic takes at least order $K^j$ many edges in the annulus, but now the standard deviation of the number of hi-mode edges in $A(j)$ which are also in the cylinder is $K^{j\frac{1+\alpha}{2}}$. This means that the expected number of decreased edges on the geodesic is at least order $K^j \cdot K^{j\frac{1+\alpha}{2}}/K^{j(1+\alpha)} = K^{j\frac{1-\alpha}{2}}$, and so when we decrease $N_j$, typically $T(0,x_n)$ will decrease by this much. If we consider only the top-level annulus (so that $K^j \sim \|x_n\|$), we obtain fluctuations at least of order $\|x_n\|^{\frac{1-\alpha}{2}}$.

This reasoning also shows why our strategy for Theorem~\ref{thm: main_thm} breaks down in dimensions $d > 2$. In the full space, the expected number of edges of decrease on the geodesic in $A(j)$ is of order $K^j/K^{dj/2}$ and this is summable in $j$. So, in total, we would typically only have finitely many edges which decrease in weight on the geodesic, and this would lead to a constant lower bound for fluctuations.

Moving to the proof, we begin as before, assuming for a contradiction:
\begin{ass}\label{ass: assumption_2}
We assume that there exist
\begin{enumerate}
\item a sequence $(w_n)$ of reals decreasing to $0$,
\item a sequence $(a_n)$ of reals, and
\item a sequence of nonzero points $(x_n)$ in $\mathbb{Z}^d$ with $\|x_n\| \to \infty$ such that if $J_n$ is defined as
\[
J_n = \left[ a_n, a_n + w_n \|x_n\|^\varpi\right]
\]
(with $\varpi = \frac{1-\alpha}{2}$), then
\[
\mathbb{P}(T(0,x_n;\alpha) \in J_n) \to 1 \text{ as } n \to \infty.
\]
\end{enumerate}
\end{ass}
Under this assumption, we define the variables $(\vec{N}, \Pi, P)$ as previously, so that they represent the number of hi-mode edges in the annuli portions $(A(k) \cap C(0,x_n;\alpha))$, a uniform ordering of the edges, and pairs of hi/lo-mode edge-weights. As already noted, only the annulus of the top-most scale will be used, but we define all the variables, to use the previous framework more easily.

\bigskip
\noindent
{\bf Step 1. Truncation and definition of good $\vec{N}$'s.} In the corresponding step 1, we define a similar truncation:
\[
T_n = \begin{cases}
A_n &\quad\text{if } T(0,x_n;\alpha) \leq A_n \\
B_n &\quad\text{if } T(0,x_n;\alpha) \geq B_n \\
T(0,x_n;\alpha) &\quad\text{otherwise},
\end{cases}
\]
except that now $B_n = A_n + \|x_n\|^{\varpi}$. The same proof as before shows that there is a choice of $A_n$ such that some median of $\mathbb{E}[T_n \mid \vec{N}]$ is equal to $M_n = (A_n+B_n)/2$.

For the definition of the ``good'' set $\mathcal{N}_n$ of values of $\vec{N}$, we first define analogous intervals $I_n,I_n'$ as
\begin{align*}
I_n&= \left[ M_n - \frac{\|x_n\|^\varpi}{2}, M_n + \frac{\|x_n\|^\varpi}{2}\right] \\
I_n'&= \left[ M_n - \frac{\|x_n\|^\varpi}{4}, M_n + \frac{\|x_n\|^\varpi}{4} \right].
\end{align*}
Also, since we focus only on the top-most scale annulus, we slightly modify the definition of $\mathcal{N}_n$.
\begin{defin}
For a sequence $(\xi_n)$ of reals converging to $0$, $\mathcal{N}_n$ is the set of those vectors $\vec{N}$ such that the following two conditions hold.
\begin{enumerate}
\item $\mathbb{P}\left( T(0,x_n;\alpha) \in I_n' \mid \vec{N}\right) > 1-\xi_n$.
\item Let $E_n = E_{n,c}$ be the event that every $T(\cdot, \cdot; \alpha)$-geodesic from 0 to $x_n$ has at least $c\|x_n\|$ hi-mode edges in $A(k_0)$, with $k_0 = k_0(n) = \lfloor \log_K \|x_n\|_\infty \rfloor$. Then
\[
\mathbb{P}\left( E_n \mid \vec{N}\right) > 1-\xi_n.
\]
\end{enumerate}
\end{defin}

As above, we have a result on the probability of $\mathcal{N}_n$:
\begin{prop}
There is a $K$ sufficiently large, a $c$ sufficiently small, and a sequence of reals $(\xi_n)$ with $\xi_n \to 0$ such that
\[
\mathbb{P}(\vec{N} \in \mathcal{N}_n) > 1-\xi_n.
\]
\end{prop}
\begin{proof}
The proof again splits into considering each item in the definition of $\mathcal{N}_n$. Item 1 is handled exactly as before (the proof of Lemma~\ref{lem: good_lemma_1} goes through without modifications). For item 2 (corresponding to Lemma~\ref{lem: good_lemma_2}), the same proof works as well. First, the analogue of \eqref{eq: geodesic_step_1} is shown in the same way, but now using that
\[
\lim_{n \to \infty} \frac{T(0,nx;\alpha)}{n} \text{ exists in probability for } x \in \mathbb{Z}^2,
\]
and is equal to the limit $g(x)$ defined in \eqref{eq: g_def}. (See \cite[Prop.~1.3]{CD}.) The proof of \eqref{eq: geodesic_step_2} is the same as before. The only other difference in the proof is that, since we consider only the annulus $A(k_0)$, in the argument below \eqref{eq: last_event_2}, we define $w_0$ and $w_1$ relative to the annuli $A(k_0-1)$ and $A(k_0+1)$. We end as previously: if $\Gamma_1$ is the segment of a $T(\cdot,\cdot;\alpha)$-geodesic $\Gamma$ (from $0$ to $x_n$) that connects $w_0$ and $w_1$, then 
\[
\Gamma_1 \text{ contains at least } \delta\|w_1\| - C_5 \|w_0\|  \text{ many hi-mode edges}.
\]
This is larger than $K^{k_0-1} \geq c\|x_n\|$, for $c$ small.
\end{proof}

\bigskip
\noindent
{\bf Step 2. Small ball for the conditional mean.} For the small ball result, the corresponding proposition is:
\begin{prop}\label{prop: new_small_ball}
There exists $\epsilon \in (0,1)$ such that
\[
\sup_{r \in I_n'} \mathbb{P}\left( \mathbb{E}[T_n \mid \vec{N}] \in [r,r+\epsilon\|x_n\|^{\varpi}], \vec{N} \in \mathcal{N}\right) \leq 1-\epsilon.
\]
\end{prop}
\begin{proof}
Again we define the variables $(X_j, \eta_j, Y_j)$ so that $N_j \stackrel{d}{=} X_j + \eta_j Y_j$, and $X_j,Y_j \geq 0$ are independent of $\eta_j$, which is a Bernoulli$(1/2)$ variable. Next, though, we need to change the definition of $\mathcal{I}_n$ to be
\[
\mathcal{I}_n = \begin{cases}
\{k_0\} & \quad\text{if } X_{k_0} \leq \mu_{k_0} - \sigma_{k_0}, Z_{k_0} \geq \mu_{k_0} + \sigma_{k_0} \\
\varnothing & \quad\text{otherwise}.
\end{cases}
\]
We then condition on $(X_i,Z_i)_{i \geq 1}$ and $(\eta_i)_{i \notin \mathcal{I}_n}$ and write the probability in the proposition as
\begin{equation}\label{eq: conditioning_again}
\mathbb{E}\left[ \mathbb{P}\left( \mathbb{E}[T_n \mid \vec{N}] \in [r,r+\epsilon \|x_n\|^{\varpi}], \vec{N} \in \mathcal{N}_n \mid (X_i,Z_i)_{i \geq 1}, (\eta_i)_{i \notin \mathcal{I}_n}\right)\right].
\end{equation}

Now we must show an analogous antichain statement:
\begin{lem}\label{lem: new_anti_chain}
There exists $\epsilon>0$ and $K$ sufficiently large such that for all large $n$, all $r \in I_n'$, and any choice of fixed $(X_i,Z_i)_{i \geq 1}$ and $(\eta_i)_{i \notin \mathcal{I}_n}$, the set
\[
Q = \left\{ (\eta_i)_{i \in \mathcal{I}_n} : \mathbb{E}[T_n \mid \vec{N}]((\eta_i)) \in [r,r+\epsilon\|x_n\|^\varpi], \vec{N} \in \mathcal{N}_n\right\}
\]
is an antichain in $\{0,1\}^{\#\mathcal{I}_n}$.
\end{lem}

Given this lemma (whose proof we will next sketch), we note that since $\mathcal{I}_n$ has cardinality at most 1, we do not need to use Sperner's theorem to bound the probability of such a small antichain. Indeed, any antichain in $\{0,1\}$ has probability at most $1/2$. Therefore, given the lemma, the probability in \eqref{eq: conditioning_again} is at most
\[
\mathbb{E}\left( \frac{1}{2} \mathbf{1}_{\{\mathcal{I}_n \neq \varnothing\}} + \mathbf{1}_{\{\mathcal{I}_n = \varnothing\}}\right).
\]
By the argument leading to \eqref{eq: bernoulli_estimate}, one has $\mathbb{P}(\mathcal{I}_n \neq \varnothing) \geq 0.04$, and so the above expectation is bounded by $1-\epsilon$. This would complete the proof of Proposition~\ref{prop: new_small_ball}.

To justify Lemma~\ref{lem: new_anti_chain}, we may suppose that $\mathcal{I}_n \neq \varnothing$ and can take $\eta^{(1)} = \eta_i^{(1)}$ and $\eta^{(2)} = (\eta_i^{(2)})$ as realizations of the $\eta$ variables that are equal off of $\mathcal{I}_n$ and such that $\eta_{k_0}^{(1)}=1$ and $\eta_{k_0}^{(2)}=0$. We follow the proof as before, but defining the event $\Upsilon$ differently:
\[
\Upsilon = \left\{ (\Pi,P) : \begin{array}{c} \text{ at least } \epsilon\|x_n\|^\varpi \text{ edges } f \in A(k_0) \cap \overline{\textsc{Geo}}^{(1)}(0,x_n) \\ \text{ have } \pi_j(f) \in (N_{k_0}^{(2)}, N_{k_0}^{(1)}]\end{array} \right\}.
\]
(Here, $\overline{\textsc{Geo}}^{(1)}(0,x_n)$ is understood as the collection of edges in any geodesic between $0$ and $x_n$ for $T(0,x_n;\alpha)$.) On the event $\Upsilon$, instead of defining only one edge $e(0)$ with the above properties, we define a sequence $e(0), \dots, e(\lfloor \epsilon \|x_n\|^\varpi \rfloor-1)$ as the first $\lfloor \epsilon \|x_n\|^\varpi \rfloor$ edges $f$ in some deterministic ordering with $\pi_{k_0}(f) \in (N_{k_0}^{(2)}, N_{k_0}^{(1)}]$.

Corresponding to Claim~\ref{clam:upslb}, we have a lower bound uniformly in $n$, $(X_i,Z_i)_{i \geq 1}$, $r \in I_n'$, $(\eta_i)_{i \notin \mathcal{I}_n}$, and $(\eta_i^{(1)})_{i \in \mathcal{I}_n} \in Q$:
\[
\mathbb{P}((\Pi,P) \in \Upsilon) \geq \delta_0 > 0.
\]
The proof of this inequality is another permutation computation. One can estimate the first and second moments of the number of edges $f \in A(k_0) \cap \overline{\textsc{Geo}}^{(1)}(0,x_n)$ with $\pi_j(f) \in (N_{k_0}^{(2)}, N_{k_0}^{(1)}]$ (conditional on $U_1 \cup U_2, P, (\pi_k)_{k \neq k_0}$ as before, and removing the event where $\#(U_1 \cup U_2) \cap \overline{\textsc{Geo}}^{(1)}(0,x_n) < c\|x_n\|$). The claimed bound follows from the Paley-Zygmund inequality.

Next, we again define $\Upsilon_1 = \{(\Pi,P) : T(0,x_n) \in I_n' \text{ in } (t_e^{(1)})\}$ and $\Upsilon_2 = \Upsilon \cap \Upsilon_1$, and decompose over the values of the $e(i)$'s (as in \eqref{eq:braket}) to obtain
\[
\sum_{(f(i))} \int_{\Upsilon_2} \left[ T_n(\vec{N}^{(1)}, \Pi, P) - T_n(\vec{N}^{(2)}, \Pi, P) \right] \mathbf{1}_{\{e(i) = f(i) \text{ for all } i\}}~\text{d}(\Pi,P).
\]
If $\hat{P}$ is the collection of all $t_e^{(L)}$ and $t_e^{(H)}$ for $e \neq f(i)$, along with $t_{f(i)}^{(H)}$ for all $i$, then Fubini's theorem again gives that the previous display is equal to
\[
\sum_{(f(i))} \int_{\Upsilon_2} \left[ \int \left[ T_n(\vec{N}^{(1)}, \Pi, P) - T_n(\vec{N}^{(2)}, \Pi, P) \right] ~\prod_i \text{d}t_{f(i)}^{(L)} \right] \mathbf{1}_{\{e(i) = f(i) \text{ for all } i\}}~\text{d}(\Pi,\hat P).
\]

The version of Claim~\ref{clam:innersanctum} we need is: there exists $c'>0$ such that for $n$ large, any choice of $f(i)$'s, any given values of the parameters of $\Upsilon$, and any choice of $\Pi, \hat P$ such that $\Upsilon_2 \cap \{e(i) = f(i) \text{ for all }i\}$ occurs,
\begin{equation}\label{eq: my_integral}
\int \left[ T_n(\vec{N}^{(1)}, \Pi, P) - T_n(\vec{N}^{(2)}, \Pi, P) \right] ~\prod_i \text{d}t_{f(i)}^{(L)} \geq c'\|x_n\|^{\varpi}.
\end{equation}
The proof is nearly the same as the corresponding proof of Claim~\ref{clam:innersanctum}. We first have
\begin{align*}
T(0,x_n;\alpha)[(\vec{N}^{(1)}, \Pi, P)] - T(0,x_n;\alpha)[(\vec{N}^{(2)}, \Pi, P)] &\geq \sum_i [t_{f(i)}^{(H)} - t_{f(i)}^{(L)}] \\
&\geq \sum_i [ d_0 - t_{f(i)}^{(L)}].
\end{align*}
As before, the integral of each $d_0 - t_{f(i)}^{(L)}$ is $\geq c''>0$, so (again using $T(0,x_n;\alpha)[(\vec{N}^{(1)}, \Pi, P)] \in I_n'$), the left side of \eqref{eq: my_integral} is at least $c'' \lfloor \epsilon \|x_n\|^{\varpi}\rfloor$. This shows \eqref{eq: my_integral}.

The end of step 2 follows the lines of its counterpart. We obtain
\[
\mathbb{E}[T_n \mid \vec{N}] (\eta^{(2)}) \leq \mathbb{E}[T_n \mid \vec{N}](\eta^{(1)}) - c'\|x_n\|^{\varpi} \mathbb{P}((\Pi,P) \in \Upsilon_2).
\]
Because $\mathbb{P}((\Pi,P) \in \Upsilon_2)$ is uniformly positive, the term on the far right is at least $\epsilon\|x_n\|^{\varpi}$ for some possibly smaller $\epsilon>0$. For this choice of $\epsilon$, we complete the proof of Lemma~\ref{lem: new_anti_chain}.
\end{proof}

\bigskip
\noindent
{\bf Step 3. Reckoning.}
The result which parallels Proposition~\ref{prop: contradiction} and contradicts Assumption~\ref{ass: assumption_2} is:
\begin{prop}
There exists $c_1>0$ and real $M_n'$ such that for all large $n$,
\[
\mathbb{P}(T(0,x_n;\alpha) \leq M_n' - c_1 \|x_n\|^{\varpi}) > c_1
\]
and
\[
\mathbb{P}(T(0,x_n;\alpha) \geq M_n' + c_1 \|x_n\|^{\varpi}) > c_1.
\]
\end{prop}
The proof of this proposition is similar to that of Proposition~\ref{prop: contradiction}, except that we have to change the centering $M_n$ to $M_n'$ because the upper bound in Proposition~\ref{prop: new_small_ball} is only $1-\epsilon$. This main inequalities to verify are then for some $M_n'$ and $c_2>0$:
\[
\mathbb{P}\left( \mathbb{E}[T_n \mid \vec{N}] \leq M_n' - c_2 \|x_n\|^\varpi \right) > c_2
\]
and
\[
\mathbb{P}\left( \mathbb{E}[T_n \mid \vec{N}] \geq M_n' + c_2 \|x_n\|^\varpi \right)  > c_2.
\]
To prove this, we argue as follows. The probability that $\mathbb{E}[T_n \mid \vec{N}]$ is in an interval of length $\epsilon \|x_n\|^{\varpi}$ centered on $M_n$ is at most $1-\epsilon$. Therefore there must be probability at least $\epsilon/2$ for the expectation to lie to the left of this interval or to the right. If it is to the left, then because $M_n$ is a median, we find that for $M_n' = M_n - (\epsilon/4) \|x_n\|^{\varpi}$, the expectation has probability at least $1/2$ to lie to the right of $M_n' + c_2 \|x_n\|^\varpi$, with $c_2 = \epsilon/4$, and probability at least $\epsilon/2$ to lie to the left of $M_n' - c_2 \|x_n\|^\varpi$. Choosing $c_2$ smaller shows the claimed bounds.

The rest of the proof follows as before, putting $M_n'$ in place of $M_n$.
\qed

\bigskip
\noindent
{\bf Acknowledgements.} The authors thank an anonymous referee for comments that helped improve the paper. The research of M. D. is supported by an NSF CAREER grant. The research of J. H. is supported by NSF grant DMS-1612921. The research of C. H. is supported by the Simons Foundation grant \#524678. Some of this work was done at the conference Recent Trends in Continuous and Discrete Probability at Georgia Tech in Summer, 2018.

\end{document}